\documentclass[times,sort&compress,3p]{elsarticle}
\usepackage[labelfont=bf]{caption}

\usepackage{graphicx}
\usepackage{float}
\usepackage{multirow}
\usepackage{subfig}
\usepackage{amsthm,amsmath,amsfonts,amssymb}
\usepackage[colorlinks,citecolor=blue,urlcolor=blue]{hyperref}
\usepackage{bigints}
\usepackage{bm}
\usepackage{color}
\usepackage[ruled,vlined]{algorithm2e}

\newtheorem{thm}{Theorem}

\newtheorem{lem}{Lemma}
\newtheorem{cor}{Corollary}

\theoremstyle{definition}
\newtheorem{defn}{Definition}
\newtheorem{rem}{Remark}
\newtheorem{ex}{Example}

\def\th@plain{%
\thm@headfont{\normalfont}
}

\newcommand{\p}{\mathbb{P}}
\newcommand{\e}{\mathbb{E}}
\newcommand{\R}{\mathbb{R}}

\newcommand{\N}{\mathbb{N}}
\newcommand{\Z}{\mathbb{Z}}

\newcommand{\bc}{\mathcal{B}}
\newcommand{\id}{\mathbf{1}}
\newcommand{\Ss}{\mathcal{S}}
\newcommand{\lk}{\left[}
\newcommand{\rk}{\right]}
\newcommand{\lc}{\left(}
\newcommand{\rc}{\right)}
\newcommand{\nlim}{\lim_{n\to\infty}}
\newcommand{\sumn}{\sum_{i=1}^n}
\newcommand{\sumd}{\sum_{i=1}^d}
\newcommand{\prodd}{\prod_{i=1}^d}

\newcommand{\suminf}{\sum_{i=1}^\infty}

\newcommand{\iinn}{{i\in\N}}
\newcommand{\leqn}{{1\leq i\leq n}}
\newcommand{\leqd}{{1\leq i\leq d}}

\newcommand{\tint}{{t\in T}}
\newcommand{\tnn}{{t\geq 0}}
\newcommand{\hto}{\lc H_t\rc_{t\geq 0}}

\newcommand{\bmt}{{\bm{t}}}
\newcommand{\bmx}{{\bm{x}}}
\newcommand{\bmX}{{\bm{X}}}
\newcommand{\bmY}{{\bm{Y}}}

\newcommand{\binfty}{{\bm{\infty}}}
\newcommand{\rmd}{\mathrm{d}}

\DeclareMathOperator{\argmax}{argmax}

\allowdisplaybreaks

\begin{document}

\begin{frontmatter}

\title{Exact simulation of continuous max-id processes with applications to exchangeable max-id sequences}
\author[1]{Florian Br\"uck \corref{mycorrespondingauthor}}
\address[1]{\footnotesize Technical University Munich, Lehrstuhl f\"ur Finanzmathematik, Parkring 11, 85748 Garching, Germany; florian.brueck@tum.de}
\date{\today}

\cortext[mycorrespondingauthor]{Corresponding author. Email address: \url{florian.brueck@tum.de}}

\begin{abstract}
    An algorithm for the unbiased simulation of continuous max-(resp.\ min-)id stochastic processes is developed. The algorithm only requires the simulation of finite Poisson random measures on the space of continuous functions and avoids the necessity of computing conditional distributions of infinite (exponent)measures. The complexity of the algorithm is characterized in terms of the expected number of simulated atoms of the Poisson random measures on the space of continuous functions. Special emphasis is put on the simulation of exchangeable max-(or min-)id sequences, in particular exchangeable Sato-frailty sequences. Additionally, exact simulation schemes of exchangeable exogenous shock models and exchangeable max-stable sequences are sketched. 
\end{abstract}


\end{frontmatter}

\section{Introduction}
This paper provides an exact simulation algorithm for real-valued continuous stochastic processes $\bmX:=\lc X_t\rc_\tint$ with the property that for every given $n\in\N$ there exist independent and identically distributed (iid) stochastic processes $\lc\bmX^{(i,n)}\rc_{1\leq i\leq n }$ such that  
\begin{align}
    \bmX\sim \max_\leqn \bmX^{(i,n)} .\label{eqnmaxidrep}
\end{align}
Such stochastic processes are called maximum-infinitely divisible (max-id) processes and they essentially constitute the class of possible weak limits of pointwise maxima of independent stochastic processes \cite{baalkemamaxidprocess1993}. Recently, max-id processes have attracted attention in the modeling of extreme events \cite{padoan2013extreme,huser2021maxidmodels,bopp2021hierarchical}, while its subclass of max-stable processes is the central object of study in the extreme value theory of iid stochastic processes.

Under the assumption that $\bmX$ and $t\mapsto \sup \{ x\in\R\mid \p\lc X_t>x\rc=1\}$ are continuous, \cite{Ginecontmaxidprocess1990,baalkemamaxidprocess1993} show that $\bmX$ can be represented as the pointwise maximum of a (usually infinite) Poisson random measure (PRM) $N=\sum_{\iinn} \delta_{f_i}$ on the space of continuous functions, i.e.,
\begin{align}
    \bmX\sim \max_{\iinn} f_i \label{eqnmaxidrepasprm}.
\end{align}
The intensity measure $\mu(\cdot):=\e\lk N(\cdot) \rk$ of the PRM $N$ is also called the exponent measure of $\bmX$ and it uniquely characterizes its distribution.
The initial motivation for our simulation algorithm for $\bmX$ stems from \cite[Algorithm 1]{dombryengelkeoestingexactsimmaxstable2016}, who have provided an exact simulation algorithm for continuous max-stable processes. In this paper, we generalize the ideas of \cite{dombryengelkeoestingexactsimmaxstable2016} to a simulation algorithm for continuous max-id processes. The key ingredient of their simulation algorithms is the PRM representation of $\bmX$ in (\ref{eqnmaxidrepasprm}) and its associated exponent measure. Basically, both simulation algorithms can be deduced from results of \cite{dombryeyiminkostrongmixing2012,dombryeyiminkoconddist2013} about the conditional distribution of a specific decomposition of the PRM $N$. This specific decomposition of the PRM $N$ allows to simulate only those functions which are relevant to determine the values of $\bmX$ at certain locations $t_1,\ldots,t_d$ and to approximate the whole sample path of $\bmX$ via the pointwise maximum over those finitely many functions. The mechanism of our simulation algorithm can be summarized as follows.
\begin{enumerate}
    \item Simulate only those functions $\lc f^{(1)}_j\rc_{1\leq j\leq k_1}$ which maximize (\ref{eqnmaxidrepasprm}) at the first location $t_1$. 
    \item For $n\in\{2,\ldots,d\}$: Given the maximizing functions at locations $t_1,\ldots,t_{n-1}$, i.e.\ $\Big\{ f^{(i)}_j \mid 1\leq i \leq n-1 ,\ 1\leq j\leq k_i\Big\}$, we only simulate those functions $\lc f^{(n)}_j\rc_{1\leq j\leq k_{n}}$ which possibly contribute to the maximum in (\ref{eqnmaxidrepasprm}) at location $t_{n}$.
    \item  Use $\hat{\bmX}=\max_{\{1\leq j\leq k_n,1\leq n\leq d\}} f^{(n)}_j$ to approximate the sample path of $\bmX$ and additionally obtain $(X_{t_1},\ldots,X_{t_d})=(\hat{X}_{t_1},\ldots,\hat{X}_{t_d})$.
\end{enumerate}

Motivated by the recent results of \cite{brueckmaischererexminidid2020} we apply the proposed simulation algorithm for continuous max-id processes to the simulation of exchangeable sequences of random variables $\bmY:=(Y_i)_\iinn$ with the property that for every $n\in\N$ there exists i.i.d.\ sequences of random variables $\bm Y:=\lc Y^{(i,n)}_j\rc_{j\in\N}$
\begin{align}
    \bm Y\sim \min_\leqn \bm Y^{(i,n)}. \label{eqnminidrep}
\end{align}
Such sequences are known as minimum-infinitely divisible (min-id) sequences and are as well characterized by a so-called exponent measure \cite{vatanmaxid1985}. It is obvious that $1/\bmY$ is a sequence of exchangeable random variables with stochastic representation (\ref{eqnmaxidrep}), therefore simply being a particular example of a general continuous max-id process with index set $T=\N$.
According to de Finetti's seminal theorem every exchangeable sequence of random variables admits the (unique) stochastic representation
\begin{align}
    \bm Y  \sim \lc \inf \big\{ t\in\R\ \big\vert \ H_t\geq E_i\big\}\rc_{i\in\N}, \label{stochrepexchseq}
\end{align} 
where $\lc E_i\rc_\iinn$ is a sequence of independent and identically distributed (iid) Exponential random variables with unit mean and $\lc H_t\rc_{t\in\R}$ denotes a (unique in law) non-negative and non-decreasing (nnnd) stochastic process with c\`adl\`ag paths.  
\cite{brueckmaischererexminidid2020} show that when $\bmY$ has the stochastic representation (\ref{eqnminidrep}) then the associated nnnd c\`adl\`ag process $H$ satisfies the property that for every given $n\in\N$ there exist iid stochastic processes $\lc H^{(i,n)}\rc_{1\leq i\leq n }$ such that
\begin{align}
    H\sim \sum_\leqn  H^{(i,n)} .\label{eqnidprocessrep}
\end{align}
Such processes are called infinitely divisible (id) and were extensively investigated in \cite{idprocessesrosinski2018}. In analogy to the L\'evy--Khintchine triplet of id random vectors on $\R^d$, id c\`adl\`ag processes are characterized by a so-called (path) L\'evy measure on the space of c\`adl\`ag functions and a deterministic c\`adl\`ag (drift-)function \cite{idprocessesrosinski2018}. 

In theory, the stochastic representation (\ref{stochrepexchseq}) immediately suggests a simulation algorithm for $\bm Y$ as the first passage times of the id process $H$ over iid Exponential barriers. In practice, however, even the approximate simulation of the associated id process $H$ is usually a challenging task. For instance, when the $d$-dimensional marginal distributions of $\bm Y$ becomes a multivariate exponential distribution \cite{marshallolkin1967}, then $H$ must belong to the class of L\'evy 
processes \cite{maischerer2009}, i.e.\ $H$ must have stationary and independent increments. Unfortunately, even for L\'evy processes, exact simulation algorithms are only known for specific families and approximate simulation algorithms are extensively discussed in the literature, e.g.\ see \cite{bondesson1982,Damien1995,asmussenrosinski2001}. Thus, the lack of the ability to simulate general processes $H$ limits the practical use of the stochastic representation (\ref{stochrepexchseq}), even though one may be able to analytically characterize the law of the id process $H$.

To overcome this challenge, we exploit a stochastic representation of $\bm Y$ in terms of minima over points of a Poisson random measure, which can be derived from (\ref{eqnmaxidrepasprm}) and the L\'evy measure and drift of the associated id process $H$.
\cite[Corollary 3.7]{brueckmaischererexminidid2020} shows that the exponent measure of $1/\bm Y$ can be uniquely characterized as a (possibly infinite) mixture of iid sequences in terms of the L\'evy measure and drift of the associated id process $H$. 
This will allow us to construct an exact simulation algorithm for $\bm Y$ via $1/\bm Y$, while essentially simulating a finite number of conditionally iid sequences. 

The rather general theoretical results about the simulation of exchangeable min-id sequences are then used to derive an exact simulation algorithm for the class of exchangeable Sato-frailty sequences, which have been fully characterized analytically in \cite{maischenkscherer2017twonovel}. Exchangeable Sato-frailty sequences can be characterized as the class of exchangeable min-id sequences associated to self-similar additive processes, i.e.\ they are associated to a stochastically continuous c\`adl\`ag processes with independent increments which have the additional property that there exists some $\gamma>0$ such that for all $a\geq 0$ $\lc H_{at} \rc_{\tnn}\sim \lc a^\gamma H_t\rc_\tnn$, see e.g.\ \cite[Section 3]{Satolevyprocess} for more details on self-similar additive processes. $H$ via (\ref{stochrepexchseq}). Even though analytical expressions of their multivariate marginal distributions are available, the simulation of such sequences has so far only been feasible for small sample sizes or some particular cases, which is due to the fact that the simulation of the associated self-similar additive process is generally complicated. We characterize the exponent measure of an exchangeable Sato-frailty sequence in terms of the L\'evy measure of the associated self-similar additive process and illustrate that our simulation algorithm essentially boils down to the simulation of two-dimensional random vectors.

In a recent article \cite{zhongexactsimmaxid2021} have independently developed a simulation algorithm for continuous max-id processes on compact non-empty real domains $T$ under the additional assumption of continuous marginal distributions. Their algorithm follows similar ideas as \cite[Algorithm 1]{dombryengelkeoestingexactsimmaxstable2016} translated to the max-id case. However, both of these algorithms require the computation of certain conditional distributions of the (infinite) exponent measure, which is usually a challenging task. Moreover, our framework is more general than that of \cite{zhongexactsimmaxid2021}, since we will explicitly consider arbitrary locally compact metric spaces $T$ as index sets and non-continuous marginal distributions. This level of generality is necessary for our purposes, since we put special emphasis on simulation algorithms for exchangeable max-id sequences $\bmX$ which have locally compact (but not compact) index sets and possibly non-continuous marginal distributions.

The remainder of the paper is organized as follows. Section \ref{sectionmaxidprocesses} summarizes the theoretical background on continuous max-id processes. Section \ref{sectionalgorithm} introduces the exact simulation algorithm for continuous max-id processes and characterizes the complexity of the algorithm. In Section \ref{sectionsimexsequence} we illustrate how our simulation algorithm for continuous max-id processes can be used to simulate exchangeable max-id sequences and we derive a particular exact simulation algorithm for exchangeable Sato-frailty sequences in Section \ref{sectionexhfrailtyseq}. Section \ref{sectionsimulation} provides a short example of how our simulation algorithm for exchangeable Sato-frailty sequences could be used in practice. \ref{sectionsimmaxidvector} provides a general exact simulation algorithm tailored to max-id random vectors. Technical lemmas and proofs can be found in \ref{apptechproofs}.

\section{Continuous max-id processes}
\label{sectionmaxidprocesses}
Let us first introduce some notation. The index set $T$ always denotes a locally compact metric space. Moreover, let $C(T):=\{ f \mid f:T\to \R \text{ is continuous} \}$ denote the space of real-valued continuous functions on $T$ equipped with the Borel $\sigma$-algebra generated by the topology of uniform convergence on compact sets. For some given function $h\in C(T)$ let $C_h(T):=\{ f \mid f\in C(T), f\geq h, f\not= h\}$ denote the space of continuous functions dominating $h$. A real-valued stochastic process defined on an abstract probability space $(\Omega,\mathcal{F},\p)$ is denoted by $\bmX:=\lc X_t\rc_\tint$. Vectors in $\R^d$ are denoted in lower case bold letters. The projection of $\bmX$ to $\bmt:=(t_1,\ldots,t_d)$ is denoted as $\bmX_\bmt:=(X_{t_1},\ldots,X_{t_d})$. The operators $\max,\min,\inf,\sup$ are always interpreted as pointwise operators, e.g.\ $\sup_{\iinn} f_i$ is interpreted as the pointwise supremum of the functions $\lc f_i\rc_\iinn$. The Dirac measure at a point $f$ is denoted as $\delta_f$. For a (random) point measure $N=\sum_{i\in\N} \delta_{f_i}$ we frequently use the notation $f\in N$ to denote that $N$ has an atom at $f$, i.e.\ to denote that $N(\{f\})\geq 1$.  
With this notation at hand we can state the definition of max-id processes and their associated vertices.
\begin{defn}[Max-id process]
A stochastic process $\bmX \in\R^T$ is called max-id if for all $n\in\N$ there exist iid stochastic processes $\lc \bm \bmX^{(i,n)}\rc_{\leqn}$ such that
\begin{align*}
    \bmX\sim \max_\leqn \bmX^{(i,n)} .
\end{align*}
The vertex of $\bmX$ is defined as the function
$$\lc h_\bmX(t)\rc_\tint :=\lc \sup \{x\in\R \mid \p\lc X_t>x\rc=1\}\rc_\tint\in[-\infty,\infty)^T.$$
\end{defn}

The most common choices for the index set $T$ of a max-id process are subsets of $\R^d$ and $\Z^d$. However, since requiring additional structure for $T$ does not yield any simplifications in the following derivations, we keep the discussion as general as possible. 

It is obvious that $g(\bmX):=\lc g\lc X_t\rc\rc_\tint$ defines a max-id process for every non-decreasing real-valued function $g$ whenever $\bmX$ is a max-id process. This implies that $\exp\lc \bmX\rc-\exp\lc h_\bmX\rc$ defines a non-negative max-id process with vertex $\bm 0$. 
In this paper, we restrict the discussion to continuous max-id processes with continuous vertex, meaning that $h_\bmX$ and $t\mapsto X_t(\omega)$ are continuous functions for every $\omega\in\Omega$. Thus, we can assume that a continuous max-id process $\bmX$ with continuous vertex is non-negative with vertex $h_\bmX=\bm 0$, since every continuous max-id processes $\bmX^\prime$ with continuous vertex $h_{\bmX^\prime}$ can be transformed to a continuous max-id process $\bmX$ with vertex $h_\bmX=\bm 0$ by setting $\bmX:=\exp\lc\bmX^\prime\rc-\exp\lc h_{\bmX\prime} \rc$.

Under the assumption of a continuous and finite vertex, \cite{Ginecontmaxidprocess1990,dombryeyiminkostrongmixing2012} have shown that a continuous max-id process $\bmX$ can be represented as the pointwise maxima of atoms of a Poisson random measure (PRM) on $C_{h_\bmX}(T)$. We summarize their results in the following theorem with the convention $\max_\emptyset :=\bm 0$.

\begin{thm}[Spectral representation of continuous max-id process \cite{Ginecontmaxidprocess1990,dombryeyiminkostrongmixing2012}]
\label{spectralrepresentationmaxidprocess} \textcolor{white}{a}
\begin{enumerate}
    \item If $\bmX$ is a continuous max-id process with vertex $h_\bmX=\bm 0$ then there exists a PRM $N$ on $C_0(T)$ with locally finite intensity measure $\mu$, called exponent measure, which satisfies
    \begin{align}
         \ \mu\lc\bigg\{ f\in C_0(T)\ \bigg\vert\ \sup_{k\in K} f(k)>\epsilon\bigg\}\rc <\infty \text{ for all compact }K\subset T\text{ and }\epsilon>0 \label{finitenessintmeasure}
    \end{align}
    such that
\begin{align*}
\bmX\sim \max_{f\in N} f .  
\end{align*}
    
\item Conversely, given a locally finite measure $\mu$ on $C_0(T)$ which satisfies (\ref{finitenessintmeasure}), there exists a PRM $N$ on $C_0(T)$ with intensity $\mu$ such that
\begin{align*}
\bmX:= \max_{f\in N}  f
\end{align*}
defines a continuous max-id process with vertex $h_\bmX=\bm 0$. 
\end{enumerate}
\end{thm}


It is easy to see that $\p\big( N(C_0(T))=\infty\big) =1$ if and only if $\mu$ is an infinite measure. For example, this is the case if $\p(X_t>0)=1$ for some $t\in T$. Since this is a desired property in many applications, a simulation of $\bmX$ via the simulation of the infinite PRM $N$ is usually practically infeasible. However, it is crucial to observe that the value of $\bmX_\bmt:=(X_{t_1},\ldots,X_{t_d})$ is fully determined by the atoms of the random measure of extremal functions at $\bmt$  
\begin{align}
    N^+_\bmt:=\sum_{f\in N} \delta_{f}\id_{\big\{ f(t_i)=X_{t_i} \text{ for some } \leqd \big\}}. \label{extremalmeasure}
\end{align} 
Thus, all atoms of the random measure of subextremal functions at $\bmt$
\begin{align}
  N^-_\bmt:=\sum_{f\in N} \delta_{f}\id_{\big\{ f(t_i)<X_{t_i} \text{ for all } \leqd \big\}} \label{subextremalmeasure}
\end{align}
are irrelevant when we are solely interested in $\bmX_\bmt$. $N^+_\bmt$, resp.\ $N^-_\bmt$, are called the extremal, resp.\ subextremal, point measure at $\bmt$. Fig.\ \ref{figprmfunction} illustrates the extremal and subextremal functions of a continuous max-id process on $\R$ with $\bmt=(0,1,\ldots,5)$.
\cite[Section 2]{dombryeyiminkoconddist2013} analyze the extremal and subextremal random point measures of a continuous max-id process and show that they are indeed well-defined. Moreover, they show that
\begin{align}
    &N^+_\bmt \textit{ is an almost surely finite random measure if and only if one of the following conditions is satisfied:} \nonumber  \\ 
     &\ (i)\ \ \ \ \ \mu(C_0(T))<\infty,\textit{ or} \nonumber \\
     &(ii)\ \ \ \ \ \mu(C_0(T))=\infty\textit{ and }\min_\leqd X_{t_i}>0\textit{ almost surely.}  \label{extremalmeasurefinitenesscondition}
\end{align}

\begin{figure}
    \centering
    \includegraphics[scale=0.5]{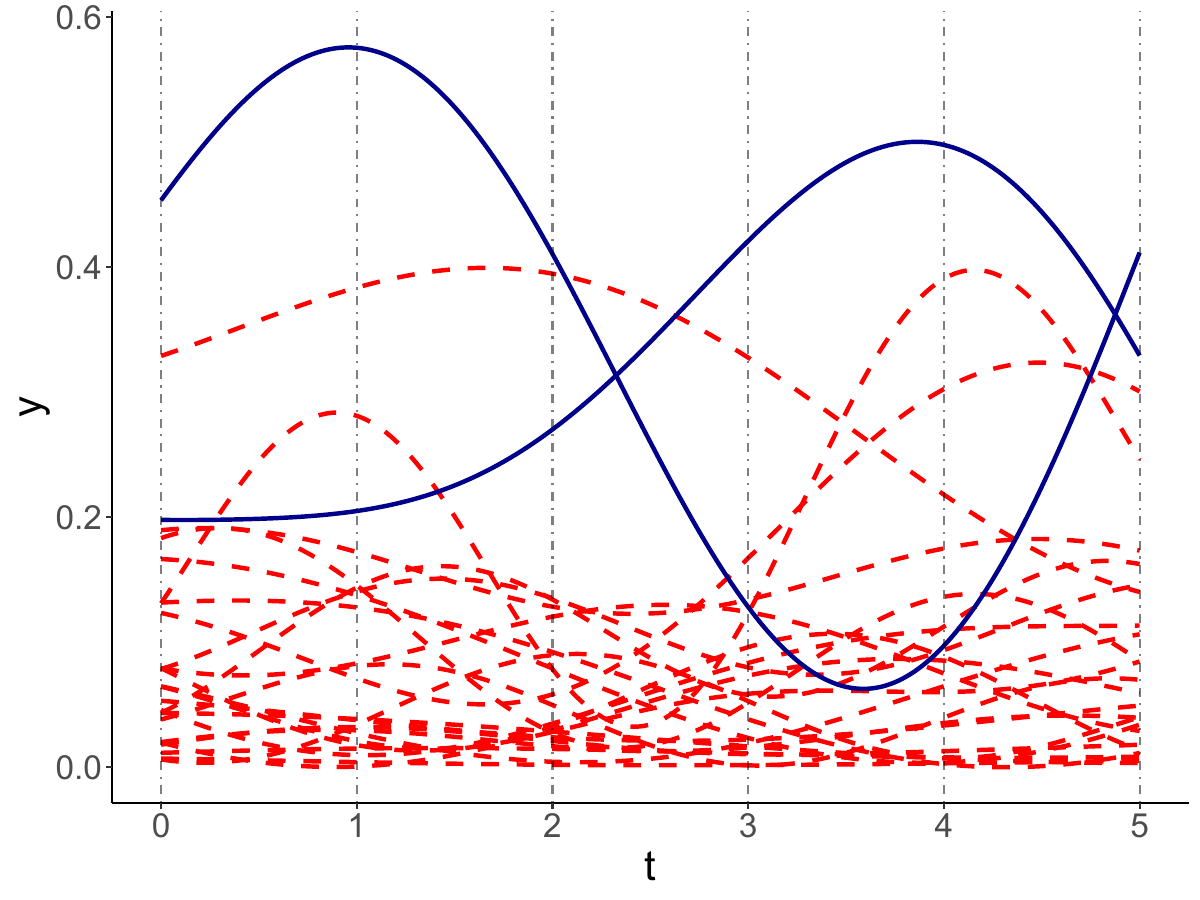}
\caption{\footnotesize  Illustration of extremal and subextremal functions of a PRM $N$. Functions in solid-blue belong to $N^+_{(0,\ldots,5)}$, functions in dashed-red belong to $N^-_{(0,\ldots,5)}$.}
\label{figprmfunction}
\end{figure}
If one of the conditions in (\ref{extremalmeasurefinitenesscondition}) is satisfied, we only need to simulate a finite number of atoms of the random measure $N^+_\bmt$ in order to obtain an exact simulation of $\bmX_\bmt$ via
$$ \bmX_\bmt=\lc \max_{f\in N^+_\bmt} f(t_1),\ldots ,\max_{f\in N^+_\bmt} f(t_d)\rc. $$
Additionally, a simulation of $N^+_\bmt$ also yields an approximation (from below) of the whole sample path of $\bmX$ via 
$$\bmX \approx\hat{\bmX}:=(\hat{X}_t)_\tint:=\lc \max_{f\in N^+_\bmt} f(t)\rc_\tint.$$
Thus, to obtain an exact simulation of $\bmX_\bmt$ and to approximate the sample path of $\bmX$ via $\hat{\bmX}$ we simply need to focus on simulation algorithms of the finite random point measure $N^+_\bmt$. 

The main ingredient of our simulation algorithm for $N^+_\bmt$ will be based on the conditional distribution of $N^-_\bmt$ given $N^+_\bmt$, which is derived in \cite[Lemma 3.2]{dombryeyiminkostrongmixing2012}. More specifically, it is shown that the conditional distribution of $N^-_\bmt$ given $N^+_\bmt$ is given by the distribution of a PRM with intensity $\id_{\{ f(t_i)<X_{t_i} , \leqd \}}\rmd \mu(f)$. To illustrate the implications of this result, let us assume we are given an initialization $N^+_{t_1}\subset N^+_\bmt\not=N$ of $N^+_\bmt$. To obtain $N^+_{(t_1,t_2)}$ we only need to consider those atoms of $N^-_{t_1}$ which belong to $N^+_{t_2}$. Given $N^+_{t_1}$, the random measure $N^+_{t_2}\setminus N^+_{t_1}$ is the restriction of $N^-_{t_1}$ to the (measurable) set 
$$\bigg\{  \Tilde{N}\textit{ extremal point measure on  }C_0(T)\textit{ at location }t_2 \textit{ and concentrated on } \Big\{f(t_2)\geq \max_{\tilde{f}\in N_{t_1}^+}  \tilde{f}(t_2) \Big\} \bigg\}.$$
Now, \cite[Lemma 3.2]{dombryeyiminkostrongmixing2012} implies that, conditional on $N^+_{t_1}$, the random measure $N^+_{t_2}\setminus N^+_{t_1}$ has the same distribution as $\argmax_{f\in \bar{N}}f(t_2)$, where $\bar{N}$ is a PRM with intensity 
\begin{align*}
\id_{\big\{ f(t_1)<\max_{\Tilde{f}\in N^+_{t_1}} \Tilde{f}(t_1) \text{ and }f(t_2)\geq \max_{\Tilde{f}\in N^+_{t_1}}\Tilde{f}(t_2) \big\}}\rmd \mu(f).  
\end{align*}
Assuming that $\max_{\Tilde{f}\in N^+_{t_1}}\Tilde{f}(t_2)$ is positive, (\ref{finitenessintmeasure}) implies that $\bar{N}$ is a \textbf{finite} PRM. Therefore, one may simulate $N^+_\bmt$ by iterative simulation of finite PRMs with intensities
\begin{equation}
\id_{\big\{ f(t_i)<\max_{\Tilde{f}\in N^+_{(t_1,\ldots,t_n)}}\Tilde{f}(t_i) \text{ for all } \leqn \text{ and } f(t_{n+1})\geq \max_{\Tilde{f}\in N^+_{(t_1,\ldots,t_n)}}\Tilde{f}(t_{n+1}) \big\}}\rmd \mu(f), \ 1\leq n\leq d-1.  \label{eqnintconditionalextremalmeasure}     
\end{equation}
From a practical perspective one should note that it is sufficient to be able to simulate from a finite PRM with intensity $\id_{\{f(t)\geq c \}}\rmd \mu(f)$ for all $t\in T$ and $c>0$ to simulate the PRMs with intensities (\ref{eqnintconditionalextremalmeasure}). To verify the claim, recall that the restriction of any PRM $\hat{N}$ with intensity $\hat{\mu}$ to an arbitrary measurable set $A$ again defines a PRM with intensity $\id_{\{f\in A\}}\rmd \hat{\mu}(f)$. Thus, to simulate a PRM with intensity (\ref{eqnintconditionalextremalmeasure}), one can simulate a PRM with intensity $$\id_{\big\{f(t)\geq \max_{\Tilde{f}\in N^+_{(t_1,\ldots,t_n)}}\Tilde{f}(t_{n+1})  \big\}}\rmd \mu(f)$$
and simply ignore those atoms which do not satisfy the constraints in (\ref{eqnintconditionalextremalmeasure}).

\begin{rem}[Infinite $N^+_{t_i}$]
\label{reminfiniteextremalmeasure}
It is easy to see that the event $X_{t_i}=0$ implies $N^+_{t_i}=N$. Thus, when $\mu$ is an infinite measure, the simulation of $N^+_{t_i}$ requires the simulation of infinitely many atoms with probability $\p\lc X_{t_i}=0\rc=\exp( -\mu( $ $ \{f\in C_0(T)\mid f(t_i)>0\}))$.
However, one may avoid this unpleasant situation by discarding finite exponent measures from $\mu$. Consider the set of possibly $0$-valued locations
\begin{align*}
   J_0:= \{ j\in \{1,\ldots,d\}\ \big\vert\ \p\lc X_{t_j}=0\rc>0\big\}
\end{align*}
and consider the exponent measures of the form 
\begin{align}
 \mu_j\lc \cdot \rc=\mu\lc \cdot \cap  \big\{ f\in C_0(T) \mid f(t_j)>0, f(t_k)=0,\ k<j,\ k\in J_0\big \}\rc,\ j\in J_0. \label{eqndefmuj}  
\end{align}
Note that the $\mu_j$ are supported on disjoint sets and that each $\mu_j$ is finite, since $0<\p\lc X_j=0\rc \leq \exp\lc-\mu_j \lc C_0(T)\rc\rc $. Therefore, it is possible to (exactly) simulate $( \hat{\bmX}_j)_{j\in J_0}$ by the simulation of PRMs with finite exponent measures $\lc \mu_j\rc_{j\in J_0}$. It remains to consider the residual of the exponent measure $\mu$, given by $\tilde{\mu}:=\mu-\sum_{j\in J_0} \mu_j$, which is more easily described as
\begin{align}
    \Tilde{\mu}\lc \cdot \rc=\mu\lc \cdot \cap  \big\{ f\in C_0(T) \mid f(t_j)=0 , j\in J_0 \big\}\rc. \label{eqndefmutilde}
\end{align}
Let $\tilde{N}$ denote a PRM with intensity $\tilde{\mu}$ and let $\tilde{\bmX}$ denote the continuous max-id process associated with the exponent measure $\tilde{\mu}$. It is not difficult to show that $\tilde{\mu}$ is either vanishing or an infinite measure and that $\tilde{\bmX}$ satisfies $\p\lc\Tilde{X}_{t_j}=0\rc=1$ for all $j\in J_0$. Therefore, the exact simulation of $\tilde{\bmX}_\bmt$ only involves the simulation of the finite random measure $\tilde{N}^+_{(t_i)_{i\not \in J_0}}$. 
Moreover, it is easily seen that $\bmX$ admits the representation
$$\bmX=\max\big\{ \Tilde{\bmX} ;\max_{j\in J_0} \hat{\bmX}_j\big\},$$
which shows that $\bmX_t$ can be determined by the pointwise maxima of finitely many finite random point measures.
\end{rem}

So far, we have assumed that we are given a finite initialization $N^+_{t_1,\ldots,t_n}$ of $N^+_\bmt$ with $\max_{f\in N^+_{t_1,\ldots,t_n}} f(t_{n+1})>0$ and, under this assumption, we have shown that we only need to simulate from finite PRMs to obtain $N^+_{t_1,\ldots,t_{n+1}}$, resp.\ $\hat{\bmX}$. In Section \ref{sectionalgorithm} we show that the ability to simulate from a PRM with intensity $\id_{\{f(t)\geq c \}}\rmd \mu(f)$ for every $t\in T$ and $c>0$ is also sufficient to obtain such initializations of $N^+_\bmt$.
Thus, we construct an algorithm for the exact simulation of $\bmX_t$ and approximation of $\bmX$ via $\hat{\bmX}$, which solely requires the ability to simulate finite PRMs with intensities $\id_{\{f(t)\geq c \}}\rmd \mu(f)$ for all $t\in T$ and $c>0$.

\section{Exact simulation of continuous max-id processes}
\label{sectionalgorithm}

The main ingredient of our algorithm is the possibility to simulate from the finite PRMs with intensities $\id_{\{f(t)\geq c\}}\rmd\mu$ for all $t\in T$ and $c>0$. 
Based on our developments in Section \ref{sectionmaxidprocesses}, we propose the following algorithm for the exact simulation of a continuous max-id process with vertex $\bm 0$. 
\begin{algorithm}
\SetAlgoLined \LinesNumbered
\KwResult{Unbiased sample of $(X_{t_1},\ldots,X_{t_d})$ and approximation of the max-id process $(X_t)_{t\in T}$.}
Set $\mu_j\lc \cdot \rc=\mu\lc \cdot \cap  \big\{ f\in C_0(T) \mid f(t_j)>0, f(t_k)=0,\ k<j,\ k\in J_0\big \}\rc ,\  j\in J_0$\;
Set $\Tilde{\mu}\lc \cdot \rc=\mu\lc \cdot \cap  \big\{ f\in C_0(T) \mid f(t_j)=0 , j\in J_0 \big\}\rc$\;
\For{$j\in J_0$}{
Simulate a finite PRM $N_j$ with intensity $\mu_j$ and set $\hat{\bmX}_j=\max_{f\in N_j}f$\;
}
Set $\tilde{\bmX}=\bm 0$\;
    \For{$i=1,\ldots,d$, $i\not\in J_0$}{
        \eIf{$\Tilde{X}_{t_i}=0$ }{
                Set $\tilde{N}^+=\emptyset$\;
                Set $c_u=\infty$ and $c_l=c$ for some $c>0$\;
            \While{ $\tilde{N}^+=\emptyset$ }{
                Simulate a finite PRM $\tilde{N}^+$ with intensity $\id_{\{c_u> f(t_i)\geq  c_l\}}\rmd \Tilde{\mu}(f)$\;
                \For{$f\in \tilde{N}^+$}{
                    \If{$f(t_k)\geq \Tilde{X}_{t_k}$ for some $ k< i$, $k\not\in     J_0$ }{
                        Set $\tilde{N}^+=\tilde{N}^+\setminus\{f\}$
                        }
                    }
                Set $c_u=c_l$ and set $c_l=c_l/2$\;
                } 
            }    
        {Simulate a finite PRM $\tilde{N}^+$ with intensity $\id_{\{f(t_i)\geq \Tilde{X}_{t_i}\}}\rmd \Tilde{\mu}(f)$\;
                \For{$f\in \tilde{N}^+$}{
                    \If{$f(t_k)\geq \Tilde{X}_{t_k}$ for some $ k< i$, $k\not\in J_0$ }{
                        Set $\tilde{N}^+=\tilde{N}^+\setminus\{f\}$
                        }
                    }
        }
        Set $\tilde{N}_{t_i}^+=\big\{f\in \tilde{N}^+ \mid f(t_i)\geq \tilde{f}(t_i) \text{ for all }\tilde{f}\in\tilde{N}^+\big\} $\;
        Set $\tilde{\bmX}=\max\big\{  \max_{f\in \tilde{N}_{t_i}^+} f,\tilde{\bmX}\big\}$\;
    }
Set $ \hat{\bmX}=\max\big\{  \max_{j\in J_0} \hat{\bmX}_j,\tilde{\bmX}\big\}$\;
 \Return{$\hat{\bmX}$}
 \caption{Exact simulation of continuous max-id process with vertex $\bm 0$} 
 \label{alg1}
\end{algorithm}

The validity of Algorithm \ref{alg1} is verified in the following theorem. 
\begin{thm}[Validity of Algorithm \ref{alg1}]
\label{thmvalalg1}
Let $\bmX$ denote a continuous max-id process with vertex $h_\bmX=\bm 0$ and exponent measure $\mu$. Then, Algorithm \ref{alg1} stops after finitely many steps and its output $\hat{\bmX}$ satisfies $\hat{\bmX}_\bmt\sim\bmX_\bmt$.
\end{thm}

Clearly, Algorithm \ref{alg1} reduces to lines $6-30$ if no margin of $\bmX_\bmt$ has an atom at $0$, since $J_0=\emptyset$ and $\tilde{\mu}=\mu$. Moreover,
it is worth mentioning that even though $\hat{\bmX}_\bmt$ is max-id, the stochastic process $\hat{\bmX}\leq \bmX$ is generally not max-id, since $N^+_\bmt$ is not a PRM on $C_0(T)$.

\begin{rem}[Reason for splitting $\mu$ into $\sum_{j\in J_0} \mu_j+\tilde{\mu}$]
The reason for splitting $\mu$ into the disjoint parts $\mu_j$ and $\tilde{\mu}$ is to divide the simulation of $\bmX$ into separate simulations of finite random measures. First, we directly simulate $\lc X_{t_j}\rc_{j\in J_0}$, i.e.\ those locations at which $\{X_{t_j}=0\}$ occurs with positive probability, since a naive simulation of $X_{t_j}$ via the respective extremal functions at $t_j$ may result in the necessity of simulating an infinite PRM with positive probability (Remark \ref{reminfiniteextremalmeasure}). Second, we simulate those atoms of a PRM $N$ with intensity $\mu$, which have not been simulated yet and possibly contribute to $\bmX_\bmt=\max_{f\in N} f(\bmt)$. Since $( \Tilde{X}_{t_j})_{j\in J_0}=\bm 0$ by the definition of $\tilde{\mu}$, we can solely focus on the simulation of $( \Tilde{X}_{t_j})_{j\not\in J_0}$. This precisely requires the simulation of the extremal functions at $(t_j)_{j\not \in J_0}$ of the PRM $\tilde{N}$ with intensity $\tilde{\mu}$.  The key observation is that the definition of the $\mu_j$ ensures that $\tilde{\bmX}$ does not have atoms at $0$ at the locations $(t_j)_{j\not \in J_0}$, which implies that the extremal point measure $\tilde{N}^+_{(t_j)_{j\not \in J_0}}$ is finite by (\ref{extremalmeasurefinitenesscondition}). Therefore, we can use (\ref{eqnintconditionalextremalmeasure}) to obtain a sample of $\tilde{N}^+_{(t_j)_{j\not \in J_0}}$ via the simulation of finite PRMs. Combining the two simulated processes by taking pointwise maxima we obtain an approximation $\hat{\bmX}$ of $\bmX$ which satisfies $\hat{\bmX}_\bmt\sim \bmX_\bmt$. 
\end{rem}

\begin{rem}[Simulation algorithm for max-stable processes \cite{dombryengelkeoestingexactsimmaxstable2016}]
\label{remalg1maxstable}
A max-stable process with unit Fr\'echet margins can be represented as $\bmX\sim \max_{i\in\N}\zeta_i \psi_i$, where $N=\sum_{i\in\N} \delta_{(\zeta_i,\psi_i)}$ is a PRM with intensity $\rmd\mu=s^{-2}\rmd s\rmd Q$ and $Q$ is a probability measure on $C_0(T)$ such that $\int f(t)\rmd Q(f)=1$ for all $\tint$. In this case, one can show that $N^+_{t}$ only contains a single function, denoted as $\hat{\psi}_t$. The regular conditional distribution of $\hat{\psi}_t$ given $X_t=z$ is obtained in \cite[Proposition 4.2]{dombryeyiminkoconddist2013}. This result can be used to represent the PRM with intensity $\id_{\{s\psi(t)>0\}}s^{-2}\rmd s\rmd Q(\psi)$ as a PRM with intensity $s^{-2}\rmd s\rmd Q_t$, where $Q_t$ denotes the conditional distribution of $\hat{\psi}_t/X_t$ given $X_t$. Thus, one may simulate a PRM with intensity $\id_{\{f(t)\geq c\}}\rmd \mu$ by successively simulating points of a PRM with intensity $\id_{\{s\geq c\}}s^{-2}\rmd s\rmd Q_t$. With this specific procedure for the simulation of a PRM with intensity $\id_{\{f(t)\geq c\}}\rmd \mu$, Algorithm \ref{alg1} essentially reduces to the exact simulation algorithm of continuous max-stable processes in \cite{dombryengelkeoestingexactsimmaxstable2016}.
\end{rem}

\begin{rem}[Conditional distribution of max-id process]
\label{remconddistmaxidprocess}
Similar to max-stable processes with unit Fr\'echet margins, \cite[Proposition 4.1]{dombryeyiminkoconddist2013} provides the conditional distribution of the extremal function of a continuous max-id process $\bmX$ with continuous marginal distributions at a location $t$, given that $X_t=z$. Intuitively, the conditional distribution can be described as the regular conditional distribution of the exponent measure $\mu$ given $X_t=z$, denoted as $Q_{t,z}$, where the formal definition of a regular conditional distribution of a possibly infinite exponent measure can be found in \cite[Appendix A2]{dombryeyiminkoconddist2013}. Thus, the extremal function for a single location $t$ can be found by first drawing a random variable $Z\sim X_t$ and then drawing the extremal function according to $Q_{t,Z}$. Surprisingly, not only the extremal function at a location $t$ given follows the conditional (on $Z$) distribution $Q_{t,Z}$, but so do the subextremal functions. More formally, assume that you are given a PRM $\sum_{i\in\N}\delta_{Z_i}$ on $(0,\infty)$ where $Z:=\max_{i\in\N} Z_i \sim X_t$. Then, conditioned on $\lc Z_i\rc_{i\in\N}$, the PRM with intensity $\id_{\{ f(t)>0\}}\rmd \mu$ can be represented as $\sum_{i\in\N} \delta_{f_{Z_i}}$, where the $f_{Z_i}\sim Q_{t,Z_i}$ are independent. \cite{zhongexactsimmaxid2021} have recently and independently proposed an algorithm for the exact simulation of max-id processes, which is based on the just described procedure to simulate a PRM with intensity $\id_{\{ f(t)>0\}}\rmd \mu$. However, determining and simulating the conditional distribution $Q_{t,Z}$ of an exponent measure is a challenging task and is only a sufficient but not a necessary criterion for the simulation of the PRM with intensity (\ref{eqnintconditionalextremalmeasure}).
\end{rem}

\begin{rem}[Simulation algorithm for max-id random vectors]
Exponent measures of max-id random vectors are often described via the geometric structure of $[-\infty,\infty)^d$, e.g.\ as scale mixtures of probability distributions on unit spheres. Examples of such families of max-id random vectors are given by random vectors with reciprocal Archimedean copula \cite{genestreciparchim2018}, max-stable distributions \cite{resnickextreme2013} and reciprocals of exogenous shock models \cite{scherersloot2019exogenous}. For these families it is generally surprisingly inconvenient to apply Algorithm \ref{alg1} due to the difficulty of describing the PRM with intensity $\id_{\{f(i)\geq c\}}\rmd \mu$ in a simple manner. Therefore, we provide a simulation algorithm which is tailored to the specific representations of exponent measures on $\R^d$ in \ref{sectionsimmaxidvector}.
\end{rem}

\subsection{Complexity of Algorithm \ref{alg1}}
\label{subseccomplexity}
The main difficulty of Algorithm \ref{alg1} lies in the simulation of the atoms (functions) of the PRMs in line 12 and 21. Therefore, to analyze the complexity of Algorithm \ref{alg1}, we may focus on the number of functions that need to be simulated to obtain $\hat{\bmX}$. Since the number of simulated functions during the execution of Algorithm \ref{alg1} is a random variable, we characterize its complexity in terms of the expected number of simulated functions. To this purpose, we extend a result by \cite{oestingschlatherzhou2013,oestingschlatherzhou2018} about the expected size of the extremal point measure of max-stable processes at locations $\bmt=(t_i)_{1\leq i\leq d}$ to max-id processes.

\begin{lem}
\label{lemexpectedsizeextremalfunction}
Let $\bmX$ denote a continuous max-id process with vertex $h_\bmX=\bm 0$ and exponent measure $\mu$. The expected size of the extremal point measure at location $\bmt$ is given by
$$\e\lk \big\vert N_{\bmt}^+\big\vert\rk=\e\lk \int_{C_0(T)} \id_{\{ f(t_i)\geq X_{t_i} \text{ for some }1\leq i \leq d \}} \rmd\mu(f)\rk.$$
\end{lem}

To deduce the expected number of simulated functions during the execution of Algorithm \ref{alg1} we additionally assume that the simulation of the atoms of a PRM with intensity $\tilde{\mu}$ can be conducted in a top-down fashion as follows: 
\begin{align}
    &\text{For all }(t_i)_{i\not\in J_0}\text{ we assume that we can consecutively simulate the atoms $\lc f^{(i)}_j\rc_{j\in\N}$ of a PRM }\nonumber \\
    &\tilde{N}=\sum_{j\in\N} \delta_{f^{(i)}_j}\text{ with intensity }\tilde{\mu}\text{ such that }f^{(i)}_1(t_i)\geq f^{(i)}_2(t_i)\geq \cdots \ .  \label{assumptionsimprm}
\end{align}
Assumption (\ref{assumptionsimprm}) allows to conduct lines 7-30 of Algorithm \ref{alg1} more efficiently: For a fixed $i\not\in J_0$ one consecutively simulates $f^{(i)}_1,f^{(i)}_2,\ldots$ such that $f^{(i)}_1(t_i)\geq f^{(i)}_2 (t_i)\geq \cdots$ and stops as soon as one has found all extremal functions at a location $t_i$. This is achieved as soon as one has found a $j\in\N$ such that $f^{(i)}_j$ is an extremal function at location $t_i$ and $f^{(i)}_{j+1}$ is subextremal function at location $t_i$. In general, it is necessary to simulate the $f^{(i)}_j$ until the first subextremal function at a location $t_i$ is found, since $\Tilde{\bmX}$ may not have continuous marginal distributions and there may be more than one extremal function at a location $(t_i)_{i\not\in J_0}$. Of course, if the distribution of $\tilde{X}_{t_i}$ is continuous, one can stop as soon as the first extremal function at location $(t_i)_{i\not\in J_0}$ is found, since there can only exist one extremal function at each continuous margin of $\Tilde{\bmX}$ by \cite[Proposition 2.5]{dombryeyiminkoconddist2013}. Thus, assumption (\ref{assumptionsimprm}) allows to avoid the simulation of more than one subextremal function at each location $(t_i)_{i\not\in J_0}$, which may happen if one conducts Algorithm \ref{alg1} in its original formulation of Theorem \ref{thmvalalg1}. 

For the remainder of this subsection we assume that Algorithm \ref{alg1} is conducted according to assumption (\ref{assumptionsimprm}). Assumption (\ref{assumptionsimprm}) may be regarded as reasonable, since it is satisfied for many continuous max-id processes. For instance, it is satisfied if one assumes that $\Tilde{\bmX}$ has continuous marginal distributions and that one conducts the simulation of the PRMs in lines 12 and 21 of Algorithm \ref{alg1} based on the conditional distribution of a max-id process as described in Remark \ref{remconddistmaxidprocess}. Moreover, the assumption may also be satisfied when simulating certain exchangeable max-id sequences, see Sections \ref{sectionexhfrailtyseq} and \ref{sectionsimulation} below. 

\begin{thm}
\label{thmcomplexityofalg1}
Under the assumption that a PRM with intensity $\tilde{\mu}$ may be simulated according to assumption (\ref{assumptionsimprm}), the expected number of the simulated functions during the execution of Algorithm \ref{alg1} is given by
$$d-\vert J_0\vert+\mu\lc \big\{ f\in C_0(T) \mid f(t_j)>0 \text{ for some } j\in J_0\big \} \rc +\sum_{ i\not\in J_0} \e\lk  \tilde{\mu}\lc \Big\{ f\in C_0(T)\mid f(t_i)\in \big[\Tilde{X}_{t_i},\infty\big) \Big\} \rc \rk .$$
Moreover, when $\Tilde{\bmX}_{(t_i)_{i\not \in J_0}}$ has continuous marginal distributions, the expected number of simulated functions during the execution of Algorithm \ref{alg1} is equal to 
$$d-\vert J_0\vert+ \mu\lc \big\{ f\in C_0(T) \mid f(t_j)>0 \text{ for some } j\in J_0\big \} \rc.$$
\end{thm}

Theorem \ref{thmcomplexityofalg1} may be interpreted as follows: The expected number of simulated functions is equal to the number of locations where $\bmX$ has continuous margins plus an additional term which accounts for the possibility that $\vert N_{t}^+\vert>1$ is possible at locations where $\bmX$ has non-continuous margins.

It is easy to see that Theorem \ref{thmcomplexityofalg1} includes the complexity characterization \cite[Proposition 9]{dombryengelkeoestingexactsimmaxstable2016} of the algorithm for simulation of continuous max-stable processes described in Remark \ref{remalg1maxstable}. There, the authors showed that the expected number of simulated functions in their algorithm is equal to the number of locations where the continuous max-stable process is simulated exactly. Theorem \ref{thmcomplexityofalg1} shows that this result also holds when Algorithm \ref{alg1} is applied to continuous max-id processes with continuous margins. Thus, when measuring simulation complexity only in terms of the expected number of simulated functions, there is no increase in simulation complexity when considering a continuous max-id processes with continuous margins. Moreover, it follows that, as a byproduct, we have shown that the expected number of simulated functions in the algorithm for exact simulation of a continuous max-id process with continuous margins and compact index set of \cite{zhongexactsimmaxid2021} is equal to the number of locations where the max-id process is simulated exactly, since it is exactly based on the assumption that the PRMs appearing in Algorithm \ref{alg1} may be simulated according to assumption (\ref{assumptionsimprm}).

\section{Exact simulation of exchangeable max(min)-id sequences}
\label{sectionsimexsequence}
When considering max-id sequences, i.e.\ $T=\N$, the assumption of continuity of the max-id process $\bmX$ is irrelevant, since $C_0(\N)=[0,\infty)^\N\setminus\{\bm 0\}$. Therefore, Algorithm \ref{alg1} is applicable to all max-id sequences with vertex $h_\bmX=\bm 0$, which may be satisfied for every max-id sequence after suitable transformations of the margins.
However, to apply Algorithm \ref{alg1}, it remains to find a suitable description of the exponent measure of a max-id sequence $\bmX$ on $[0,\infty)^\N\setminus\{\bm 0\}$ such that the PRM with intensity $\id_{\{f(i)\geq c\}}$ can be simulated. To achieve this, we focus on the results of \cite{brueckmaischererexminidid2020}, who describe the structure of exponent measures of exchangeable min-id sequences, i.e.\ of exchangeable sequences $\bmY:=1/\bmX$, where $\bmX$ is max-id. To this purpose, let us recall the most important results of \cite{brueckmaischererexminidid2020}.
\begin{thm}[\text{\cite[Corollary 3.7]{brueckmaischererexminidid2020}}]
    $\bmY\in(0,\infty]^\N$ is an exchangeable min-id sequence if and only if \begin{align}
    \bmY \sim \lc \inf\{t\geq 0 \mid H(t)\geq E_i\}\rc_\iinn,    \label{correspexminidseqidprocess}
    \end{align}
    where $\lc E_i\rc_\iinn$ are iid $Exp(1)$ and $H=\lc H_t\rc_{t\geq 0}\in [0,\infty]^{[0,\infty)}$ is a (unique in law) nnnd id  c\`adl\`ag process which satisfies $H_0=0$. 
\end{thm} 
We say that an exchangeable max-id sequence $\bmX$ corresponds to an id process $H$ if and only if $\bmY=1/\bmX$ is the exchangeable min-id sequence corresponding to $H$. Similar to max-id sequences, the survival function of $\bmY$ can be expressed in terms of an exponent measure $\bar{\mu}$. It can be related to the exponent measure of $\bmX$ noting that $\p(\bmY>\bmx)=\p\lc \bmX<\frac{1}{\bmx}\rc=\exp\lc-\mu\lc\lc-\binfty,\frac{1}{\bmx}\rc^\complement \rc\rc =\exp\lc-\bar{\mu}\lc (\bmx,\binfty]^\complement\rc\rc $, where $\bar{\mu}(A):=\mu\lc \{\bmx\in [0,\infty)^\N \mid 1/\bmx \in A\}\rc$ is called the exponent measure of the exchangeable min-id sequence $\bmY$. From this relation it is easy to see that a PRM $\bar{N}=\suminf \delta_{\bm f_i}$ with intensity $\bar{\mu}$ can be transformed to a PRM $N=\suminf \delta_{1/\bm f_i}$ with intensity $\mu$. Thus, we can generate atoms of $N$ by taking reciprocals of atoms of $\bar{N}$. In the following, we will show how the correspondence (\ref{correspexminidseqidprocess}) can be used to generate atoms from $\bar{N}$ (and thus from $N$).

Let us recall several facts about id processes. It is well-known that id random vectors on $[0,\infty]^d$ are in one-to-one correspondence with a pair $(\upsilon,b)$, where $\upsilon$ is a (L\'evy)measure on $[0,\infty]^d\setminus\{\bm 0\}$ satisfying certain integrability conditions and $b\in[0,\infty]^d$ is a deterministic (drift)vector. \cite{idprocessesrosinski2018} has elegantly extended this characterization to id processes and \cite{brueckmaischererexminidid2020} have used these results to prove that the Laplace-transform of a nnnd id c\`adl\`ag process which satisfies $H_0=0$ is given by
\begin{align}
    \e\lk \exp\lc -\sumd a_i H_{t_i}\rc\rk=\exp\lc-\sumd a_i b(t_i) -\int_{M} 1-\exp\lc-\sumd a_i g(t_i)\rc \rmd \nu(g)\rc,\ \bm a,\bmt\in[0,\infty)^d, \label{eqnlaplacetrafoidprocess}
\end{align}
where $\nu$ is a unique (L\'evy)measure on the path space 
$$\mathbf{M}:=\big\{g:[0,\infty)\to [0,\infty] \mid g(0)=0,  g\text{ nnnd and c\`adl\`ag},\ g\not=\bm 0\big\},$$
which satisfies $\int_{\mathbf{M}} \min \{1,g(t)\}\rmd\nu(g)<\infty$ for all $t\geq 0$ and $b\in \mathbf{M}$ is a unique deterministic (drift)function.
From (\ref{eqnlaplacetrafoidprocess}) and the formula for the Laplace transform of a PRM, see e.g.\ \cite[Section 3]{resnickextreme2013}, one can deduce that 
$$H\sim b+ \hat{H}$$
may be decomposed into a deterministic drift $b$ and a ``completely random'' process $\lc\hat{H}_t\rc_{t\geq 0}\sim \lc \int_{\mathbf{M}} g(t)\rmd N_H(g) \rc_{t\geq 0}$, where $N_H:=\sum_\iinn \delta_{g_i}$ is a PRM on $\mathbf{M}$ with intensity measure $\nu$.

Combining (\ref{correspexminidseqidprocess}) and (\ref{eqnlaplacetrafoidprocess}) we obtain that
\begin{align*}
\p\lc\bmY>\bmx\rc=\e\lk\exp\lc-\suminf H_{x_i}\rc\rk=\exp\lc-\suminf  b\lc x_i \rc -\int_{{\mathbf{M}}} 1-\exp\lc-\suminf g\lc x_i\rc\rc \rmd \nu(g)\rc.    
\end{align*}
This shows that the exponent measure $\bar{\mu}$ of the exchangeable min-id sequence $\bmY$ is given by
$$ \bar{\mu}\lc A\rc = \bar{\mu}_b (A) +\int_{\mathbf{M}} \otimes_{i=1}^\infty \big( 1-\exp\lc -g(\cdot)\rc\big) (A) \rmd\nu(g), $$
where 
\begin{itemize}
    \item $\bar{\mu}_b$ denotes the exponent measure of an iid sequence with stochastic representation $\lc \inf\Big\{t\geq 0\mid  b(t)\geq E^{(2)}_i\Big\}\rc_{i\in\N}$ and marginal distribution function $1-\exp(-b(\cdot))$, 
    \item $\otimes_{i=1}^\infty \big( 1-\exp\lc -g(\cdot)\rc\big)$ denotes the distribution of an iid sequence with marginal distribution function $1-\exp( -g(\cdot))$ and $\int_{\mathbf{M}} \otimes_{i=1}^\infty \big( 1-\exp\lc -g(\cdot)\rc\big) (A) \rmd\nu(g)$ denotes the exponent measure of an exchangeable max-id sequence with stochastic representation $\lc \inf\Big\{t\geq 0\mid  \hat{H}(t)\geq E^{(1)}_i\Big\}\rc_{i\in\N}$
\end{itemize}
and $\lc E^{(j)}_i\rc_{\iinn}, j\in\{1,2\}$ denote sequences of iid Exp$(1)$ distributed random variables. In other words, $\bar{\mu}$ is the sum of the exponent measure of an iid sequence and a possibly infinite mixture of iid sequences. Since addition of two exponent measures stochastically corresponds to applying component-wise minima to two independent min-id sequences, we get
$$ \bm Y \sim  \min \bigg\{\lc \inf\Big\{t\geq 0\mid  b(t)\geq E^{(2)}_i\Big\}\rc_{i\in\N} ; \lc \inf\Big\{t\geq 0\mid  \hat{H}(t)\geq E^{(1)}_i\Big\}\rc_{i\in\N}\bigg\},$$
which shows that the only difficulty in the simulation of $\bm Y$ is the simulation of the sequence $\lc \inf\Big\{t\geq 0\mid \hat{H}(t)\geq E^{(1)}_i\Big\}\rc_{i\in\N}$.  Hence, for our analysis, we can ignore the presence of $\bar{\mu}_b$, i.e.\ assume that $b=0$, and focus on the simulation of the sequence $\lc \inf\Big\{t\geq 0\mid  \hat{H}(t)\geq E^{(1)}_i\Big\}\rc_{i\in\N}$ with exponent measure of the form
\begin{align}
    \bar{\mu}\lc A\rc = \int_{\mathbf{M}} \otimes_{i=1}^\infty \lc 1-\exp( -g(\cdot))\rc (A) \rmd\nu(g). \label{exponentmeasuremaxidsequence}
\end{align} 

Representation (\ref{exponentmeasuremaxidsequence}) implies that the atoms of the PRM $\bar{N}$ with intensity $\bar{\mu}$ can be generated as follows:
\begin{enumerate}
    \item[(i)] Generate a PRM $N_H=\sum_{i\in\N} \delta_{g_i}$ on $\mathbf{M}$ with intensity measure $\nu$,
    \item[(ii)]  For each $g_i$, draw an iid sequence $\bar{\bm f}_i$ with distribution function $ 1-\exp\big( -g_i(\cdot)\big)$,
    \item[(iii)]  Set $\bar{N}=\sum_{i\in\N} \delta_{\bar{\bm f}_i}$.
\end{enumerate}
As mentioned previously, a PRM $N$ with intensity $\mu$ is then obtained by taking the reciprocal of each atom of $\bar{N}$, i.e.\ by defining $N=\sum_{i\in\N} \delta_{1/\bar{\bm f}_i}$. Therefore, a PRM with intensity $\id_{\{f(n)\geq c\}}\rmd \mu$ can be generated by the simulation of a finite PRM with intensity 
\begin{align}
    &\bar{\mu}\lc\cdot\cap\  \{f(n)\leq 1/c\}\rc=\int_{\mathbf{M}} \otimes_{i=1}^\infty \big( 1-\exp\lc -g(\cdot)\rc\big) \lc \cdot\cap \{  f(n)\leq 1/c\}\rc \rmd\nu(g) \nonumber  \\
    &=\int \otimes_{i=1}^\infty \big( 1-\exp\lc -g(\cdot)\rc\big)  \lc \big\{ \bm f\in\ \cdot\ \big\}\ \big\vert\ f(n)\leq 1/c  \rc \Big( 1-\exp\big( -g\lc 1/c\big)\rc\Big)    \rmd\nu(g). \label{intfiniteprm}
\end{align} 
It is important to observe that $\lc 1-\exp(-g(1/c))\rc   \rmd\nu(g)$ defines an exponent measure with total finite mass 
\begin{align}
    C_c:=-\log\bigg( \e\big[ \exp\lc -H_{1/c}\rc\big] \bigg)=-\log\big(\p\lc X_1< c\rc\big). \label{eqnconstantcc}
\end{align}
Thus, to simulate from a PRM with intensity $\id_{\{f(n)\geq c\}}d\mu$ it is sufficient to be able to simulate from the probability measure $P_c:=C_c^{-1}\lc 1-\exp(-g(1/c))\rc \rmd\nu(g)$ on $\mathbf{M}$. Since the measure $\nu$ can be chosen rather arbitrarily it is hopeless to expect a general recipe for the simulation of $P_c$. However, there are many families of stochastic processes for which $\nu$ can be conveniently described such that simulation from $P_c$ becomes feasible. One of these families is the class of self-similar additive processes (\cite{JEANBLANC2002},\cite[Section 3]{Satolevyprocess}), which will be investigated in the next section.

\section{Exchangeable Sato-frailty sequences}
\label{sectionexhfrailtyseq}
Choosing $\lc H_t\rc_{t\geq 0}$ in (\ref{correspexminidseqidprocess}) as a non-negative and non-decreasing additive process, also called additive subordinator, gives rise to the class of so-called exchangeable exogenous shock models \cite{maischenkscherer2016,scherersloot2019exogenous}. Exchangeable exogenous shock models are characterized by the property that every $d$-dimensional margin $\bmY_d$ of $\bmY$ can be stochastically represented as the minimum of independent random shocks, each of them affecting a certain subset of components of $\bmY_d$. More formally, every $d$-dimensional margin of the exchangeable exogenous shock model $\bmY$ can be represented as 
\begin{align}
\bmY_d\sim\lc \min  \big\{E_I \ \big\vert \ I\subset\{1,\ldots,d\},\ i\in I\big\} \rc_\leqd, \label{exexshockrep}
\end{align}
where the shocks $\lc E_I\rc_{I\subset\{1,\ldots,d\}}$ denote independent non-negative random variables with continuous distribution function and $E_{I_1}\sim E_{I_2}$ if $\vert I_1\vert=\vert I_2\vert$. Moreover, the distribution of the $E_I$ is uniquely linked to the Laplace transform of the associated additive subordinator $H$ \cite{maischenkscherer2016,scherersloot2019exogenous}. In principle, the results of \cite{maischenkscherer2016,scherersloot2019exogenous} could be used to simulate the $d$-dimensional margins of an exchangeable exogenous shock model. However, even if the Laplace transform of the associated additive subordinator is known analytically, it is numerically challenging to compute the distribution of the individual shocks $E_I$ and $2^d$ random variables have to be simulated to determine $\bmY_d$. Thus, if $d$ is large it is practically infeasible to simulate an exchangeable exogenous shock model via the representation (\ref{exexshockrep}). Alternatively, one could use the representation (\ref{correspexminidseqidprocess}) to generate a sample of $\bmX$, which circumvents the curse of dimensionality. Unfortunately, the simulation of the additive subordinator $H$ is usually infeasible or only possible approximatively. Therefore, an exact and efficient simulation of high dimensional exchangeable exogenous shock models has remained an open problem to date. 

A subclass of exchangeable exogenous shock models has been investigated in \cite{maischenkscherer2017twonovel} by restricting $\lc H_t\rc_{t\geq 0}$ to the class of self-similar subordinators (aka Sato subordinators), meaning that $H$ is an additive subordinator and that there exists some index $\gamma>0$ such that for all $a\geq 0$ $\lc H_{at}\rc_{t\geq 0}\sim \lc a^\gamma H_t\rc_{t\geq 0}$. The exchangeable sequences associated to self-similar subordinators are called exchangeable Sato-frailty sequences. \cite[Section 3]{Satolevyprocess} shows that every self-similar additive process $H$ with index $\gamma$ is uniquely associated to its distribution at unit time. The law of $H_1$ belongs to the class of self-decomposable distributions, meaning that for every $c\in(0,1)$ there exists a random variable $H^{(c)}$ independent of $H_1$ such that $H_1\sim cH_1+H^{(c)}$. Self-decomposable laws constitute a broad subclass of infinitely divisible distributions, e.g.\ containing the (inverse-)Gaussian, Laplace, (tempered-)stable, Fr\'echet, Pareto, Exponential and (inverse-)Gamma distribution as well as several laws appearing in financial modeling as the CGMY, Normal Inverse Gaussian and Meixner distribution \cite{barndorffnielsennigdistr1996,schoutensteugels1998meixner, carrgemanmadanyor2002cgmymodel} to provide some examples. Furthermore, every self-decomposable distribution can be obtained as the law of a self-similar additive process at unit time. Thus, there is a one-to-one correspondence of self-similar additive subordinators with index $\gamma$, non-negative self-decomposable distributions and the class of exchangeable Sato-frailty sequences. \cite[Proposition 16.5]{Satolevyprocess} shows that the index $\gamma$ of a self-similar process can be changed to an arbitrary index $\tilde{\gamma}$ via the simple time change $t\mapsto t^{\tilde{\gamma}/\gamma}$. Combined with \cite[Corollary 3.2]{brueckmaischererexminidid2020}, which shows that a time-change of the self-similar subordinator corresponds to the marginal transformation $a\mapsto a^{-\tilde{\gamma}/\gamma}$ of the associated exchangeable Sato-frailty-sequence, we can w.l.o.g.\ assume that $\gamma=1$ to simplify further derivations.

The key quantity of our simulation algorithm will be the univariate L\'evy measure of the self-decomposable distribution of $H_1$. It allows us to derive a convenient representation of path L\'evy measure of the associated self-similar subordinator, which then translates into a simple representation of the exponent measure of the associated exchangeable Sato-frailty sequence via (\ref{exponentmeasuremaxidsequence}). To this purpose, we recall several characterizations of self-decomposable laws, which are provided in \cite[Section 3]{Satolevyprocess}. \cite[Theorem 15.10]{Satolevyprocess} shows that the L\'evy measure $\upsilon$ of a non-negative self-decomposable distribution is absolutely continuous w.r.t.\ the Lebesgue measure with density of the form $\rmd \upsilon =k(a)a^{-1}\id_{\{a>0\}} \rmd a$, where $k$ is some non-increasing right-continuous function such that $\int_0^\infty \min\{a,1\} k(a)a^{-1}\rmd a<\infty$. 

Noting that $k$ defines a measure $\rho_k$ on $(0,\infty)$ by $\rho_k \big((a,\infty)\big):=k(a)$, $a>0$, we can rewrite $\upsilon$ as $\upsilon(\rmd a)=\rho_k  \big((a,\infty)\big) a^{-1}\rmd a$. It turns out that $\rho_k$ defines the L\'evy measure of another non-negative id distribution \cite[Theorem 17.5]{Satolevyprocess}. Thus, there exists a non-negative and non-decreasing L\'evy process, also called L\'evy subordinator, $\lc L^{(k)}_t\rc_{t\geq 0}$ with univariate L\'evy measure $\rho_k$ and the self-decomposable distribution with L\'evy measure $\upsilon$ can be recovered from $L^{(k)}$ as the distribution of the infinitely divisible random variable $\int_0^\infty \exp(-s)\rmd L^{(k)}_s$. Due to this representation $L^{(k)}$ is called the Background Driving L\'evy process (BDLP) of the self-decomposable distribution with L\'evy measure $\upsilon$. \cite{JEANBLANC2002} show that not only the self-decomposable distribution associated with L\'evy measure $\upsilon$, but also the associated self-similar subordinator $H$ can be recovered from (the law of) $L^{(k)}$ by 
\begin{align}
    \hto := \lc \int_{-\log\lc \min\{t; 1\}\rc}^\infty \exp(-s)\rmd L^{(k,1)}_s+\int_0^{\log\lc \max\{1,t\}\rc} \exp(s)\rmd L^{(k,2)}_s\rc_{t\geq 0} ,\label{selfsimsubfrombldp}
\end{align} 
where $L^{(k,i)}_{i=1,2}$ denote two iid copies from $L^{(k)}$. This particular representation of the self-similar subordinator $H$ allows us to derive a representation of its associated path L\'evy measure in terms of the L\'evy measure of the BDLP.

\begin{lem}[L\'evy measure of self-similar subordinator via L\'evy measure of BDLP]
\label{lemlevymeasureselfsimbdlp}
Let $\hto$ denote a self-similar subordinator with index $1$ and let $L^{(k)}$ denote the BDLP associated to $H_1$. The L\'evy measure of $H$ can be expressed as
\begin{align}
    \nu(A)&= \int_{(0,\infty)^2} \id_{ \big\{  as\id_{\{ \cdot \geq s\}}  \in A  \big\}} s^{-1}\rmd s\otimes \rho_k (\rmd a)  ;\ A\in\bc({\mathbf{M}}), \label{pathlevymeasureselfsimBDLP}
\end{align}
where $\rho_k$ denotes the L\'evy measure of $L^{(k)}$.
\end{lem}

Assuming that $k(\cdot)=\rho_k\big((\cdot,\infty)\big)$ is differentiable we obtain the following corollary.

\begin{cor}[L\'evy measure of self-similar subordinator via density]
\label{lemlevymeasureselfsilevysystemdensity}
Let $\hto$ denote a self-similar subordinator with index $1$ associated to the self decomposable distribution with L\'evy measure $\rmd\upsilon=k(a)a^{-1}\rmd a$. If $k$ is differentiable, then the path L\'evy measure $\nu$ of $\hto$ is given by 
$$ \nu(A)=-\int_{(0,\infty)^2} \id_{ \big\{  a\id_{\{\cdot\geq s\}}\in A\big\}} k^\prime\lc as^{-1}\rc s^{-2} \rmd s\rmd a,\ A\in\bc({\mathbf{M}}) .$$
\end{cor}

Having determined the L\'evy measure of a self-similar subordinator we can express the exponent measure of the associated exchangeable Sato-frailty sequence by (\ref{exponentmeasuremaxidsequence}) as 
$$\bar{\mu}(A)=-\int_0^\infty\int_0^\infty \otimes_{i=1}^\infty \lc 1-\exp\lc-a\rc \id_{\{\cdot \geq s\}}\rc \lc \big\{ (h_i)_{i\in\N}\in A\big\} \rc k^\prime\lc as^{-1}\rc s^{-2} \rmd a\rmd s ,$$
assuming that $k(\cdot)=\rho_k\big((\cdot,\infty)\big)$ is differentiable.
Thus, to apply Algorithm \ref{alg1}, we need to simulate a PRM with intensity (\ref{intfiniteprm}) expressible as
\begin{align*}
    -\int_0^\infty \int_0^{1/c} \otimes_{i=1}^\infty \lc 1-\exp\lc-a\id_{\{\cdot\geq s\}}\rc\rc  \lc \big\{ f\in\ \cdot\ \big\}\ \bigg\vert\ f(n)\leq\frac{1}{c}   \rc  k^\prime\lc as^{-1}\rc s^{-2}\lc 1-\exp(-a)\rc    \rmd s\rmd a.
\end{align*}
The only difficulty in the simulation of this PRM is the simulation of the random vector $(A^{(c)},S^{(c)})$ with joint distribution $-\id_{\{s\in(0,1/c)\}}\id_{\{0<a\}}C_c^{-1} k^\prime\lc as^{-1}\rc s^{-2}\lc 1-\exp(-a)\rc  \rmd a\rmd s$, where $C_c$ was defined in (\ref{eqnconstantcc}). However, it is easy to see that the marginal density of $A^{(c)}$ is given by 
\begin{align}
    g_{A^{(c)}}(a)=\id_{\{ a>0\}}C_c^{-1} a^{-1}k(ac)\lc 1-\exp(-a)\rc \label{densitygA}
\end{align}
and that the conditional density of $S^{(c)}$ given $A^{(c)}$ is given by 
\begin{align}
    g_{S^{(c)}|A^{(c)}}(s)= -\id_{\{s\in(0,1/c)\}}k^\prime\lc A^{(c)}s^{-1}\rc s^{-2}\lc  \lc A^{(c)}\rc^{-1}k\lc A^{(c)}c\rc \rc^{-1} \label{densitygSgivenA}.
\end{align}
Thus, a sample of $(A^{(c)},S^{(c)})$ can be generated by first sampling a random variable $A^{(c)}$ with density $g_{A^{(c)}}$ and then, given $A^{(c)}$, sampling a random variable $S^{(c)}$ according to the conditional density $g_{S^{(c)}|A^{(c)}}$. Altogether, this implies that we can sample from a PRM with intensity (\ref{intfiniteprm}) by the following procedure:
\begin{enumerate}
    \item[(i)]  Draw a random variable $M\sim Poi(C_{c})$.
    \item[(ii)]  For $1\leq i\leq M$ draw a random variable $A^{(c)}_i$ according to the density $g_{A^{(c)}}$ and, conditioned on $A^{(c)}_i$, draw a random variable $S^{(c)}_i$ according to the density $g_{S^{(c)}|A^{(c)}_i}$.
    \item[(iii)]  For each pair $(A^{(c)}_i,S^{(c)}_i)$ draw an iid sequence $\bm f^{(A^{(c)}_i,S^{(c)}_i)}:=\lc f^{(A^{(c)}_i,S^{(c)}_i)}_j\rc_{j\in\N}$ with marginal distribution $\big( 1-\exp\lc-A^{(c)}_i\rc\big) \delta_{S^{(c)}_i}+ \exp(-A^{(c)}_i)\delta_\infty $ and set $f^{\lc A^{(c)}_i,S^{(c)}_i\rc}_n=S^{(c)}_i$.
    \item[(iv)]  The PRM with intensity (\ref{intfiniteprm}) is given by $\sum_{i=1}^M \delta_{ 1/\bm f^{\big( A^{(c)}_i,S^{(c)}_i\big)}}$.
\end{enumerate}
Therefore, Algorithm \ref{alg1} can be employed to generate exchangeable Sato-frailty sequences if the associated function $k$ is differentiable. When $k$ is not assumed to be differentiable, one still obtains a representation of $\bar{\mu}$ in terms of $\rho_k$ via Lemma \ref{lemlevymeasureselfsimbdlp}. However, the simulation procedure of a PRM with intensity (\ref{intfiniteprm}) slightly changes and requires the simulation of a random vector with density  $\id_{\{s\in(0,1/c)\}}\id_{\{0<a\}}C_c^{-1}s^{-1}(1-\exp(-as))\mathrm{d} s\mathrm{d} \rho_k(a)$, which cannot be conducted without assuming further regularity properties of $\rho_k$.

\begin{rem}[Sampling of the densities $g_{A^{(c)}}$ and $g_{S^{(c)}|A^{(c)}}$]
If the function $k$ is known analytically one can use rejection sampling to obtain (exact) samples from random variables with density $g_{A^{(c)}}$, see e.g.\ \cite[p.\ 235 ff.]{maischerer2017simulating} for more details on rejection sampling. To simulate a random variable with density $g_{S^{(c)}|A^{(c)}}$ one could also use rejection sampling if $k^\prime$ is known analytically, but one should notice that its associated distribution function is given by $G_{S^{(c)}|A^{(c)}}(s)=k(A^{(c)}s^{-1})A^{(c)} k\lc A^{(c)}c\rc^{-1}\id_{\{0<s<1/c\}}$. Therefore, rejection sampling and the (numerical) inverse transform sampling method may be used to sample random variables with conditional density $g_{S^{(c)}|A^{(c)}}$.
\end{rem}

\begin{rem}[Extension to exchangeable exogenous shock models]
\cite{idprocessesrosinski2018} shows that an id-process $H$ is additive if and only if its path L\'evy measure is concentrated on one-time jump functions of the form $a\id_{\{\cdot\geq s\}}$. In many cases, the univariate L\'evy measure $\upsilon_t$ of an extended real-valued additive process at time $t$ is absolutely continuous w.r.t.\ to the Lebesgue measure, meaning that $\upsilon_t(\rmd a)=k(a,t)\rmd a$. If $k(a,\cdot)$ is differentiable on $(0,\infty)$ for almost all $a$ one can obtain a similar expression of the path L\'evy measure of an additive process as in Corollary \ref{lemlevymeasureselfsilevysystemdensity}.  It can be easily checked that the image measure of the map $\big(  (0,\infty) \times (-\infty,\infty),k^\prime(a,s)\rmd a\rmd s \big)\to {\mathbf{M}};\ (s,a)\mapsto a\id_{\{\cdot\geq s\}}$ satisfies the conditions of \cite[Theorem 2.8]{idprocessesrosinski2018} and thus defines a valid L\'evy measure of a driftless additive process $\tilde{H}$. We obtain that 
\begin{align*}
    \e\lk \exp\lc iz\tilde{H}_t\rc\rk&=\exp\lc ib(t)+\int_{\R\setminus \{0\}} \int_0^\infty \lc \exp\lc iza\id_{\{t\geq s\}}\rc -1-iza\id_{\{t\geq s\}}\id_{\{a\leq 1\}}  \rc k^\prime (a,s) \rmd s\rmd a \rc\\
    &=\exp\lc ib(t)+  \int_{\R\setminus \{0\}} \int_0^t \lc  \exp\lc iza\rc -1-iza\id_{\{a\leq 1\}} \rc  k^\prime (a,s)  \rmd s\rmd a \rc\\
   &= \exp\lc ib(t)+ \int_{\R\setminus \{0\}}  \lc \exp\lc iza\rc -1-iza\id_{\{a\leq 1\}} \rc  k(a,t) \rmd a \rc= \e\lk \exp\lc izH_t\rc\rk ,
\end{align*}
since $\lim_{s\to 0} k(a,s)=0$ for almost all $a\not = 0$ by the stochastic continuity of additive processes. Thus, $\tilde{H}$ and $H$ are identical in distribution, given that additive processes are uniquely determined by their marginal distributions. Therefore, the path L\'evy measure of $H$ is given by 
$$\nu (A) =\int_{\R\setminus \{0\}} \int_0^\infty \id_{\big\{ a\id_{\{\cdot\geq s\}}\in A  \big\}} k^\prime(a,s) \rmd a\rmd s,\ A \in \bc\lc D\big([0,\infty)\big)\rc ,$$
where $D\big([0,\infty)\big)$ denotes space of real-valued c\`adl\`ag functions.
Similar to Sato-frailty sequences, this allows to sample the exchangeable sequences associated to an additive subordinator with univariate L\'evy measure $\upsilon_t(\rmd a)=\id_{\{a>0\}}k(a,t)\rmd a$ by repeatedly drawing random vectors $(A^{(c)},S^{(c)})$ and conditionally iid sequences.
\end{rem}

\begin{rem}[Extension to exchangeable max-stable sequences]
Stochastic processes $H$ which satisfy $H(nt)\sim \sumn $ $H^{(i)}(t)$ for all $n\in N$ and iid copies $\lc H^{(i)}\rc_\iinn$ of $H$ are called strongly infinitely divisible w.r.t.\ time (strong-idt). \cite{maischerer2014exmsmve} have shown that the max-id sequence $\bmX$ corresponding to a strong-idt process in (\ref{correspexminidseqidprocess}) is max-stable, meaning that its marginal distributions can be obtained as a limit distribution of scaled maxima of iid random vectors. The general form of the exponent measure of an exchangeable max-stable sequence $\bmX$ has been derived in \cite{mai2020canonicalspectral}. However, it still involves the law of a stochastic process and does not directly translate into a simple simulation procedure for Algorithm \ref{alg1}.  \cite{BernhartMaiScherer2015,maischerer2019strongidtubordinators} have investigated subfamilies of strong-idt processes with the particular representation $\lc H_f(t)\rc_\tnn=\lc \int_0^\infty f(s/t) \rmd L_s\rc_\tnn$, where $\lc L_t\rc_{\tnn}$ denotes a L\'evy subordinator and $f$ denotes a non-negative non-increasing left-continuous function. 
\cite{maischerer2019strongidtubordinators} provide exact simulation algorithms for the $d$-dimensional margins of the corresponding exchangeable max-stable sequence in the particular case $f(s)=\lim_{u\nearrow s}-\log(F(u))$ for some distribution function $F$, whereas the models in \cite{BernhartMaiScherer2015} could only be simulated when $L$ is a compound Poisson process.
Lemma \ref{lemlevymeasurestochasticintegration} yields a rather simple representation of the exponent measure of the exchangeable max-stable sequence associated to $H_f$ in terms of the path L\'evy measure of $L$, which can be translated into a representation of the exponent measure of $\bmX$ as a mixture of iid sequences in terms of the law of a random vector $(A^{(c)},S^{(c)})$. Thus, similar to Sato-frailty sequences, the examples from \cite{BernhartMaiScherer2015,maischerer2019strongidtubordinators}  can essentially be simulated by repeated simulations of a random vector $(A^{(c)},S^{(c)})$ and conditionally iid sequences.
\end{rem}

\section{Illustration of the simulation algorithm}
\label{sectionsimulation}

In this section we exemplarily demonstrate how Algorithm \ref{alg1} can be used to simulate a Sato-frailty sequence in practice. We rather aim at providing a proof-of-concept like exposition than to fine-tune the presented example to its most efficient simulation procedure. We chose the Inverse Gaussian
(IG) distribution as our guiding example. The IG distribution is known to be self-decomposable \cite{halgreen1979self} and its L\'evy measure is given by
\begin{align*}
    \upsilon (\rmd a)= \id_{\{a>0\}} \frac{\delta}{\sqrt{2\pi} } a^{-\frac{3}{2}} \exp\lc -\frac{\gamma^2 a}{2} \rc \rmd a, \text{ where }\delta,\gamma>0.
\end{align*}
Therefore, the L\'evy measure of the associated self-similar subordinator is characterized by the function $$k(a)=\frac{\delta}{\sqrt{2\pi} } a^{-1/2} \exp\lc -\gamma^2 a/2 \rc,\ a>0.$$ 
One should note that a simulation of the associated exchangeable Sato-frailty sequence $\bmX$ via its stochastic representation (\ref{correspexminidseqidprocess}) would either require the simulation of the whole path of the infinitely active BDLPs in (\ref{selfsimsubfrombldp}) or the direct simulation of the increments of the associated self-similar subordinator. However, the simulation of the whole path of the BDLPs cannot be practically achieved nor can the increments of the associated self-similar subordinator be efficiently simulated, since their law cannot be easily characterized. Thus, Algorithm \ref{alg1} can be seen as the natural choice regarding the simulation of $\bmX$. It is quite easy to see that the associated random vector $(A^{(c)},S^{(c)})$ can be simulated by rejection sampling for the random variable $A^{(c)}$ with density (\ref{densitygA}) and inverse transform sampling for the random variable $S^{(c)}|A^{(c)}$ with conditional density (\ref{densitygSgivenA}).
We have simulated the corresponding sequence $\bmX$ for various values of $(\delta,\gamma)$ and report our results in terms of scatterplots of the associated copula $C(u_1,\ldots,u_d)=\p\lc F_1(X_1)\leq u_1,\ldots , F_d(X_d)\leq u_d\rc$, since copulas do not depend on the marginal distribution of $\bmX$. In particular, the associated copula is independent of the index of self-similarity of the associated self-similar subordinator. The copula corresponding to $\bmX$ has been analytically derived in \cite{maischenkscherer2017twonovel} and is given by
$$ C(u_1,\ldots,u_d)= \prodd \exp\lc \delta\gamma \lc \sqrt{1+i\lc \frac{\log\lc u_{[i]}\rc}{\delta\gamma} +1\rc^2 -i }-\sqrt{1+(i-1)\lc \frac{\log\lc u_{[i]}\rc}{\delta\gamma}+1 \rc^2 -(i-1) }\rc \rc ,$$
where $u_{[i]}$ is defined as the $i$-th oder statistic of $(u_1,\ldots,u_d)$. Thus, the associated copula only depends on $\delta\gamma$. Fig.\ \ref{figurecopulainvgaus} provides the empirical copula plots for dimensions $d\in\{2,3\}$ and $\delta\gamma\in\{1/10,2,10\}$ and shows that the margins of $\bmX$ become less dependent with increasing $\delta\gamma$.
\begin{figure}%
    \centering
    \subfloat{\includegraphics[width=6cm]{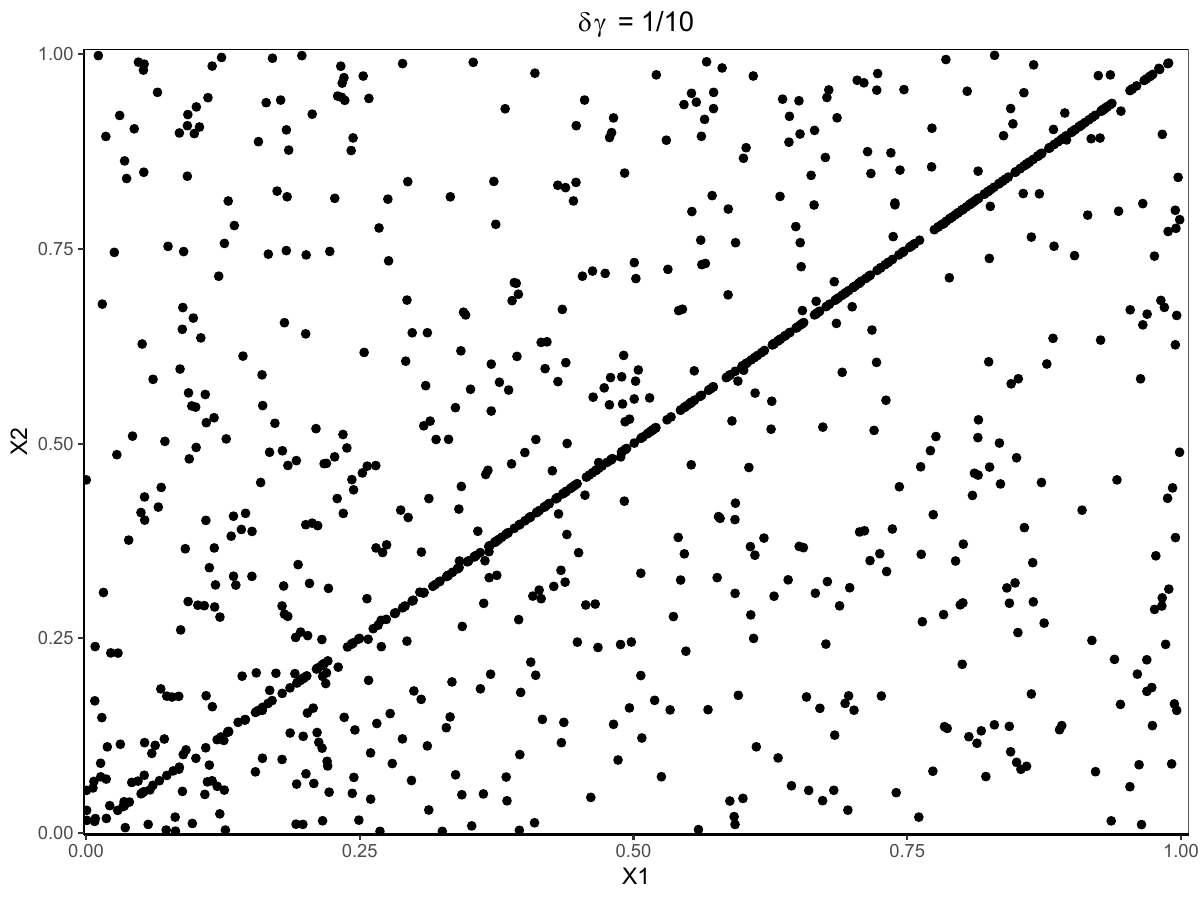} }%
    \subfloat{\includegraphics[width=7cm]{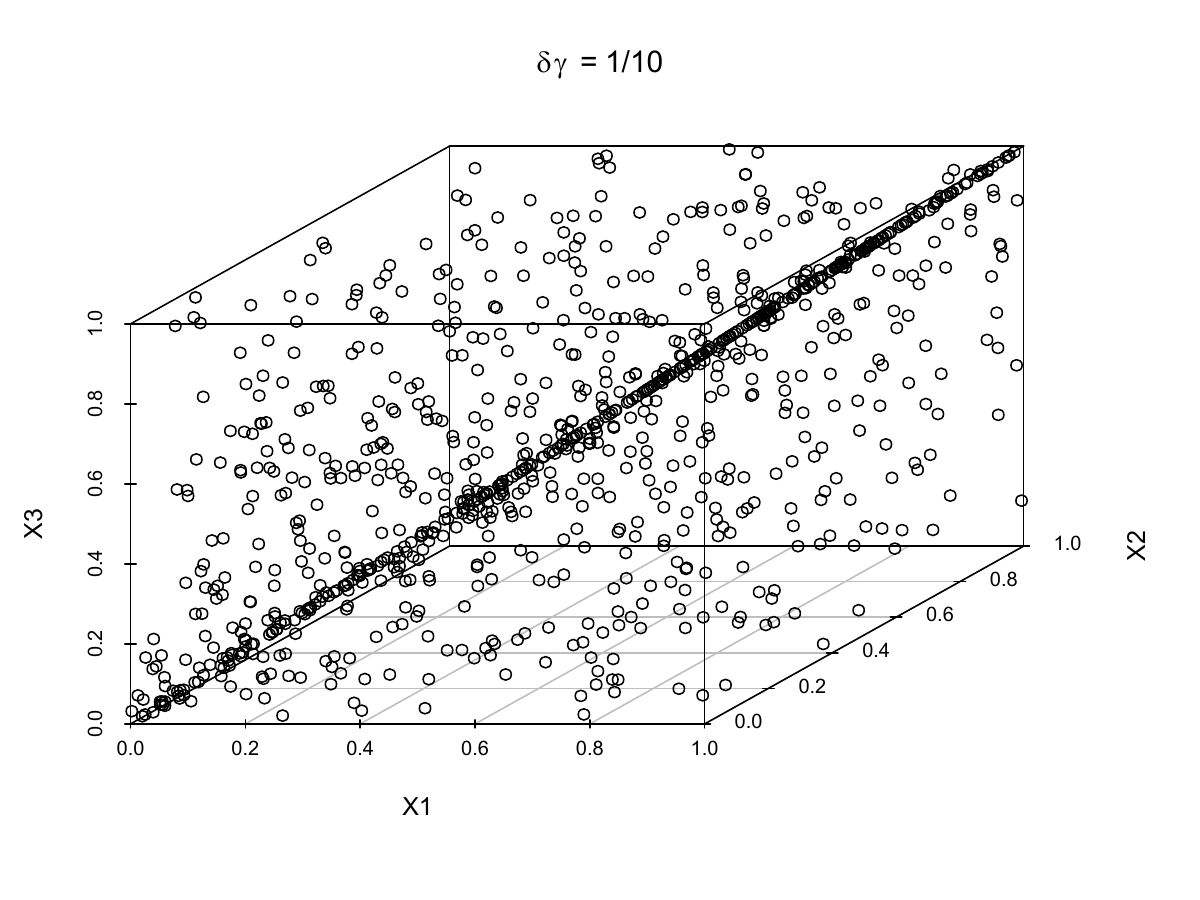} }%

    \subfloat{\includegraphics[width=6cm]{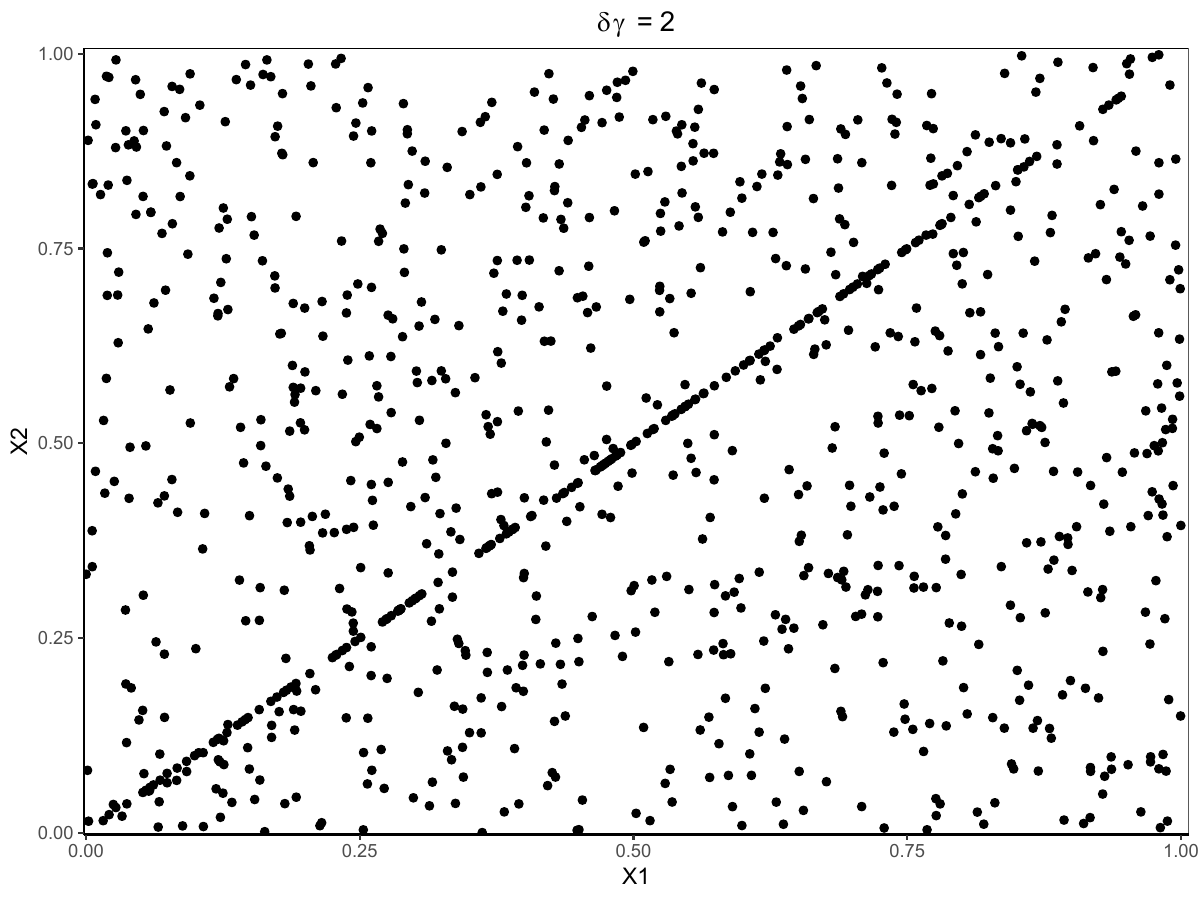} }%
  \subfloat{\includegraphics[width=7cm]{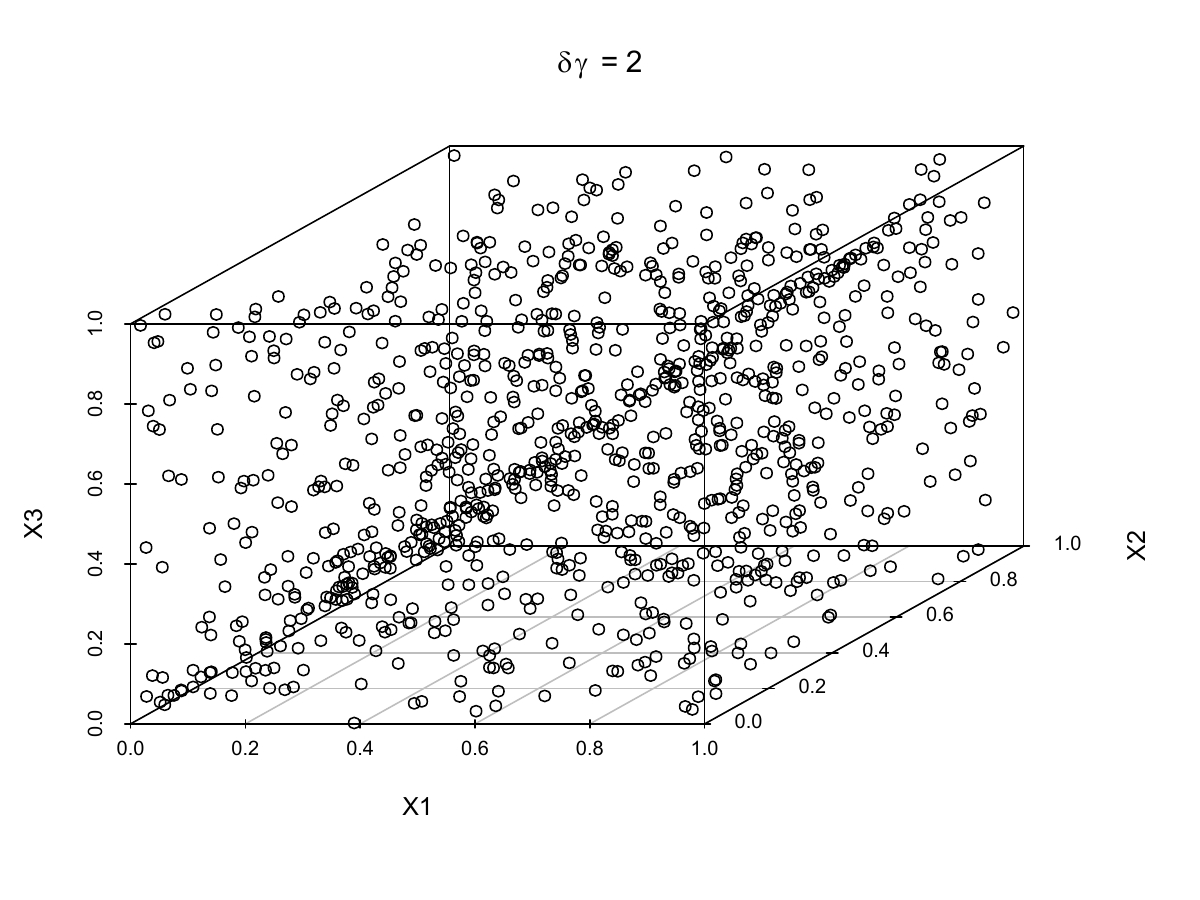} }%

     \subfloat{\includegraphics[width=6cm]{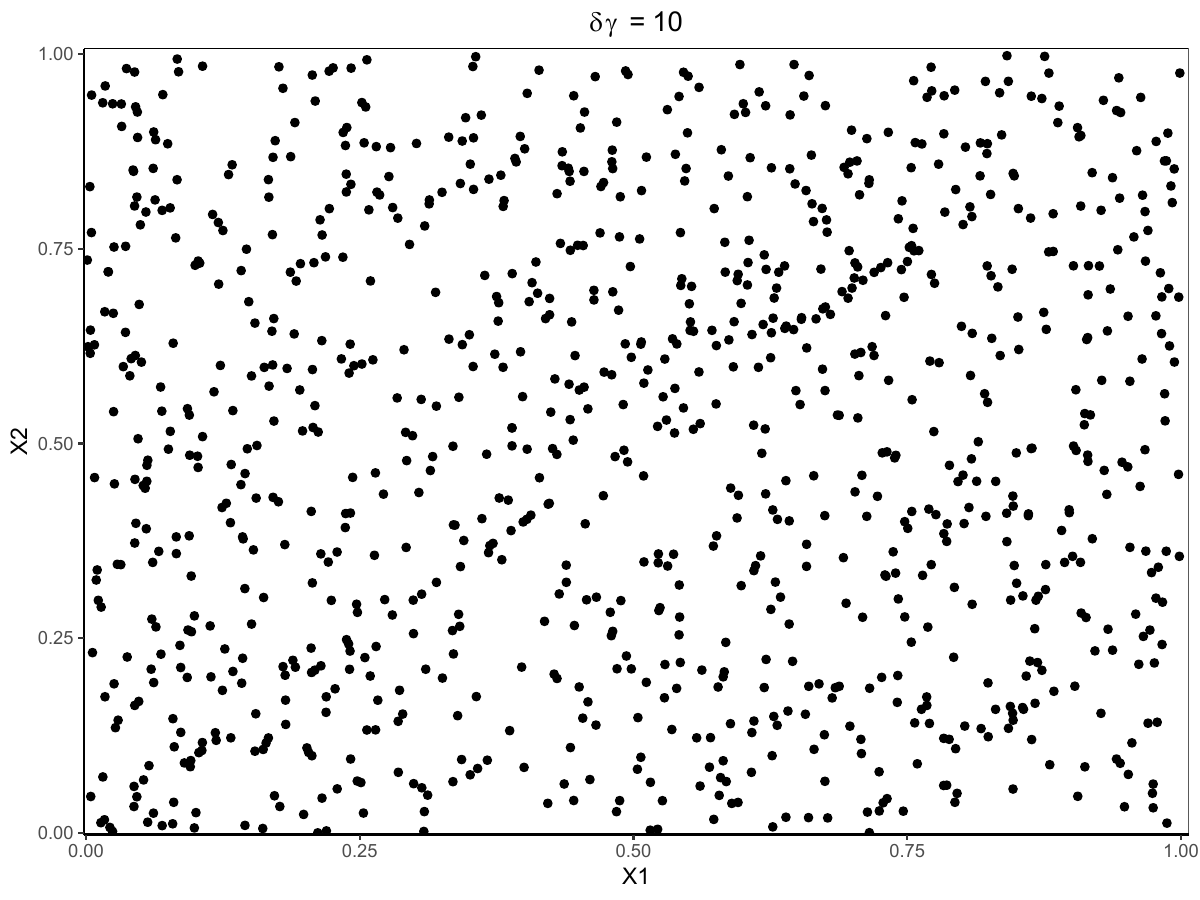} }%
       \subfloat{\includegraphics[width=7cm]{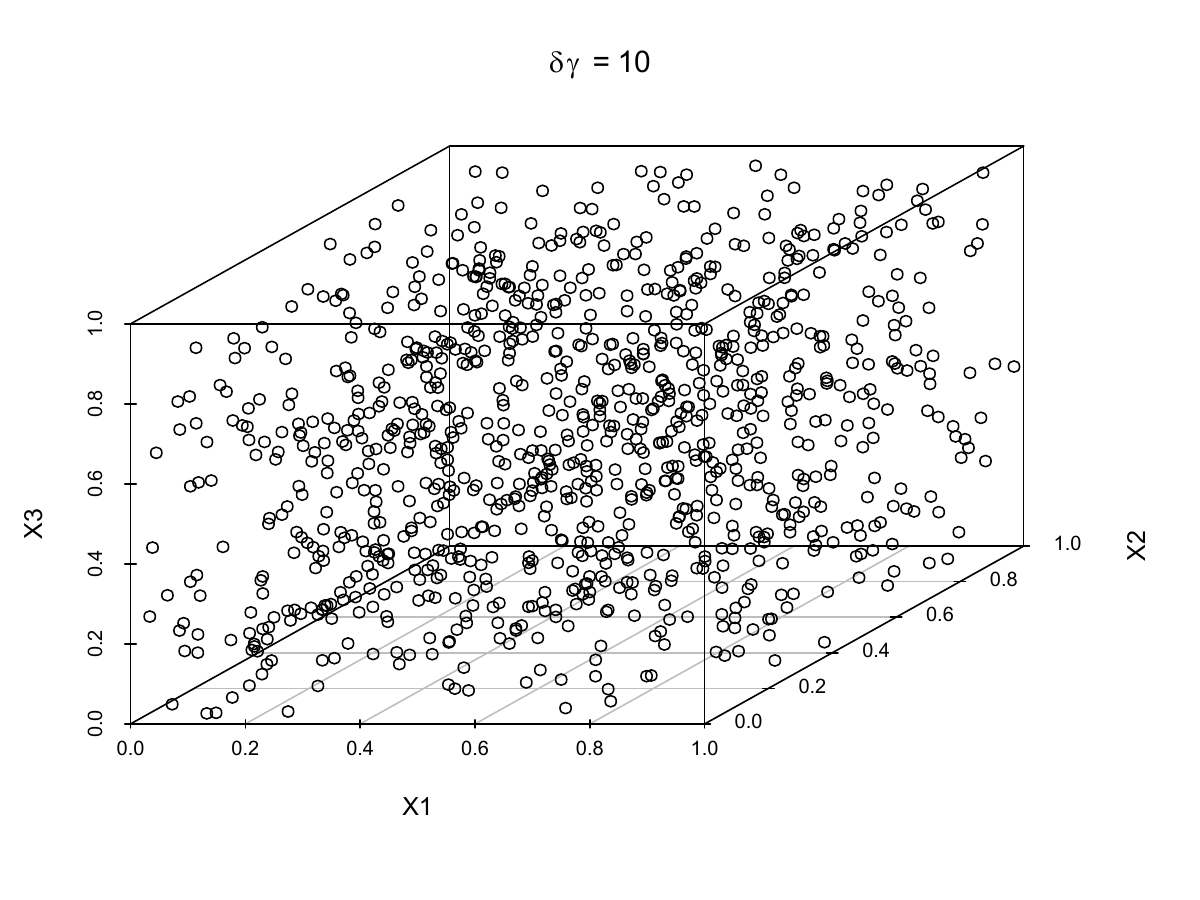} }%
    \caption{Empirical copula plots of the two- and three-dimensional margins of $1000$ samples of the exchangeable Sato-frailty sequences associated to the IG distribution with parameters $\delta\gamma\in\{1/10,2,10\}$ (top, middle, bottom). }%
    \label{figurecopulainvgaus}%
\end{figure}

To empirically verify the results of Theorem \ref{thmcomplexityofalg1}, Fig.\ \ref{fignumsimfct} shows the average and standard deviation of the number of simulated sequences to produce a sample of the exchangeable Sato-frailty sequences associated to the IG distributions with $\delta\gamma\in\{1/10,2,10\}$ and various dimensions $d\in\{1,5,10,25,50,100,250,500,1000,2500,5000,10000\}$ over $500$ repetitions. Note that we applied Algorithm \ref{alg1} in accordance with assumption (\ref{assumptionsimprm}) as follows: After simulating the random vectors $\lc \lc A_i^{(c)},S_i^{(c)}\rc \rc_{1\leq i\leq M}$ where $M\sim Poi(C_c)$ via the (conditional) densities (\ref{densitygA}) and (\ref{densitygSgivenA}), we can simulate the sequences associated to the $\lc \lc A_i^{(c)},S_i^{(c)}\rc \rc_{1\leq i\leq M}$ in increasing order of the $S_i^{(c)}$. It is easy to see that this procedure allows to simulate a PRM with intensity $\mu$ according to assumption (\ref{assumptionsimprm}). Therefore, the expected number of simulated sequences is equal to $d$, since the one-dimensional marginal distributions of exchangeable Sato-frailty sequences are continuous. Fig.\ \ref{fignumsimfct} empirically verifies this result, showing that the average number of simulated functions over $500$ repetitions is always close to $d$, independently of $\delta\gamma$ and $d$. Interestingly, the standard deviation of the number of simulated functions seems to depend on $\delta\gamma$. Thus, the example shows that, even though the expected number of simulated sequences (or functions) in Algorithm \ref{alg1} is always equal to $d$ when the margins of $\bmX$ follow a continuous distribution, its standard deviation may depend on properties of the associated continuous max-id process. 

\begin{figure}
    \centering
    \subfloat{\includegraphics[width=7cm]{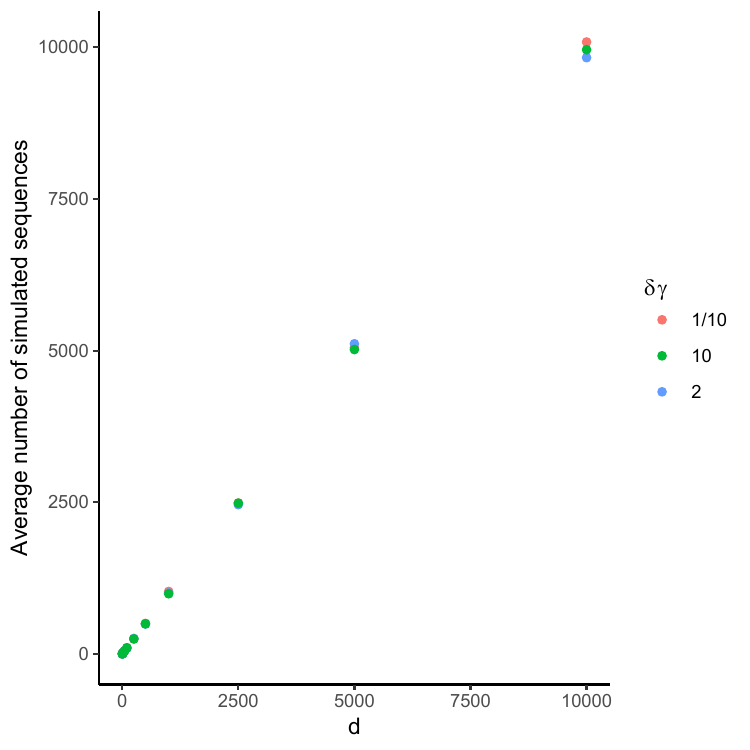} }%
    \subfloat{\includegraphics[width=7cm]{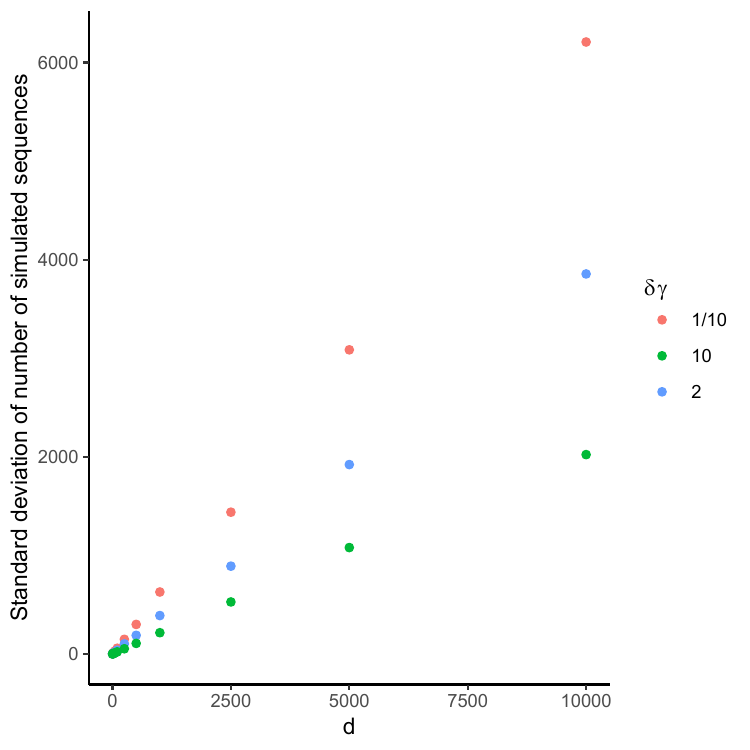} }
    \caption{The average (left) and standard deviation (right) of the number of simulated sequences to obtain one sample of the exchangeable Sato-frailty sequences associated to the IG distribution for various values of $\delta\gamma$ and dimensions $d$ over $500$ repetitions.}
    \label{fignumsimfct}
\end{figure}


\section{Discussion}
We have developed a simulation algorithm for the exact simulation of continuous max-id processes based on their associated exponent measure. Our algorithm is solely based on the ability to simulate PRMs with finite intensity measures which facilitates its wide applicability. The complexity of the algorithm has been characterized in terms of the expected number functions that need to be simulated to obtain a sample of the associated continuous max-id process. Exemplarily, we have derived the exponent measure of an exchangeable Sato-frailty sequence and demonstrated the applicability of our algorithm theoretically and in practice, thereby providing the first exact simulation algorithm for high dimensional samples of this family. We have sketched how the algorithm for exchangeable Sato-frailty sequences can be generalized to certain families of exogenous shock models and max-stable sequences without increasing its practical complexity. This enables the possibility to consider the construction principle of exchangeable max-id sequences in (\ref{stochrepexchseq}) by means of its desired analytical properties, without the need of having a suitable simulation algorithm for the associated stochastic process at hand. \ref{sectionsimmaxidvector} discusses an alternative simulation algorithm for max-id random vectors, thereby accounting for the natural geometric descriptions of many known families of finite dimensional exponent measures. An application of our algorithm to a continuous max-id process with uncountable index set is left for future research. We think that the proposed simulation algorithm may be extended to upper semicontinuous max-id processes with obvious modifications, however the technical details need to be carefully worked out and are also left for future research.

\appendix

\section{Exact simulation of max-id random vectors}
\label{sectionsimmaxidvector}
This section is devoted to the exact simulation of a max-id random vector $\bmX\in [0,\infty)^d$. Since $\bmX$ can be viewed as a continuous max-id process on $T=\{1,\ldots,d\}$ Algorithm \ref{alg1} is, in principle, applicable to every max-id random vector.  However, the exponent measure of a max-id random vector is often more easily described by exploiting the specific geometric structure of $[0,\infty)^d$. For example, a common representation of an exponent measure of a max-id random vector is the scale mixture of a probability distribution on the non-negative unit sphere of some norm on $\R^d$. Two famous representatives of this class of exponent measures are the exponent measures of max-stable random vectors with unit Fr\'echet margins \cite[Chapter 5]{resnickextreme2013} and random vectors with reciprocal Archimedean copula \cite{genestreciparchim2018,maiexactsimreciparchim2018}, see Example \ref{examplescalemixturesimplexponentmeasure} below. In both cases, a simulation of $\bmX$ via Algorithm \ref{alg1} would require to deviate from the natural description of the exponent measure to simulate a PRM with intensity $\id_{\{ f(t)\geq c\}}\rmd\mu(f)$. Thus, there is a need to adapt Algorithm \ref{alg1} to exploit the natural structure of many exponent measures of max-id random vectors. Again, to simplify the theoretical developments, we can w.l.o.g.\ assume that $h_\bmX=\bm 0$.

Our goal is to generalize the algorithms of \cite{schlathermodelsmaxstable2002,dombryengelkeoestingexactsimmaxstable2016,maiexactsimreciparchim2018} to max-id random vectors. Similar to Algorithm \ref{alg1} we will only simulate those atoms of the PRM $N=\sum_{i\in\N} \delta_{\bmx_i}$ with intensity $\mu$ which may be relevant to determine $\max_{i\in \N} \bmx_i=\bmX$. We start by dividing $[0,\infty)^d$ into disjoint ``slices'' $\Ss_n$ of finite $\mu$-measure. Then, assuming that we can simulate finite PRMs $N_n$ with intensities $\mu\lc \cdot\cap \Ss_n\rc$, we iteratively simulate the $N_n$ until a stopping criterion is reached. To obtain a valid stopping criterion we need to assume that the slices $\Ss_n$ eventually approach $\bm 0$, which is mathematically described as eventually residing in an open ball around $\bm 0$. This will force the algorithm to stop after finitely many steps, since atoms of the PRM $N$ in a neighborhood of $\bm 0$ eventually cannot contribute to the maximum of the already simulated points. 

\begin{ex}
Assume that the atoms of the PRM $N$ are given by the points in Fig.\ \ref{figprmdim2}. A possible execution of our algorithm could be described as follows: 
In the first step, all atoms of the PRM above the blue line are simulated, which corresponds to $\Ss_1=\{\bmx\mid x_1+x_2\geq 0.6 ,x_1,x_2\geq 0\}$. Since the pointwise maximum of atoms of $N$ above the solid-blue line is not above the dashed-blue line $\{\bmx\mid \min_{i=1,2} x_i= 0.6\}$ there are possibly some atoms of $N$ which can contribute to the pointwise maximum of $N$ and which have not yet been simulated. Therefore, in a second step, we simulate all points between the solid-blue and the solid-green line, which corresponds to $\Ss_2=\{\bmx\mid 0.3\geq x_1+x_2< 0.6 ,x_1,x_2\geq 0\}$. The red triangle denotes the pointwise maximum of the simulated points above the solid-green line. Since it is above the dashed-green line $\{\bmx\mid \min_{i=1,2} x_i= 0.3\}$ it is the maximum of the PRM $N$ and the algorithms stops.
\begin{figure}
        \centering
    \includegraphics[scale=0.4]{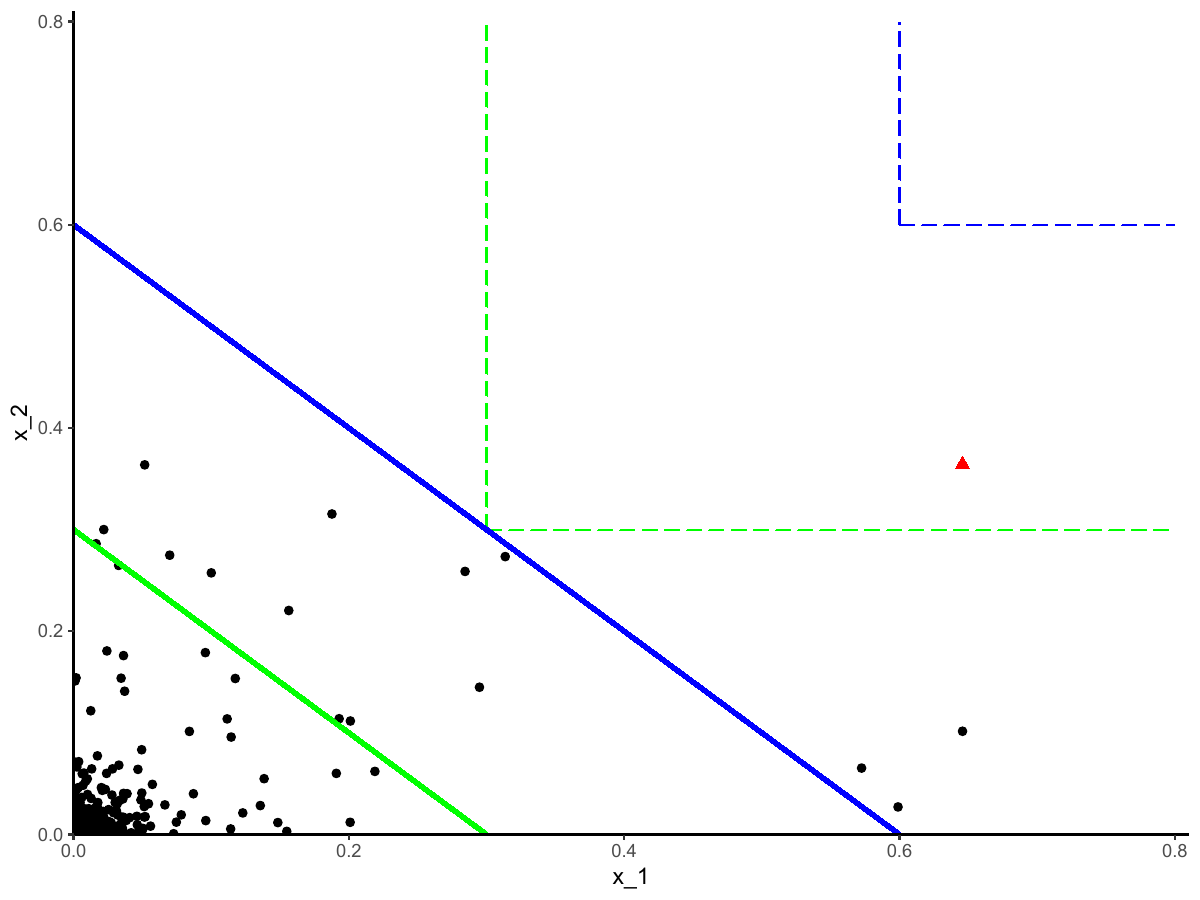}
    \caption{\footnotesize  Illustration of Algorithm \ref{alg2}.}
    \label{figprmdim2}
\end{figure}

\end{ex}

Let us formalize the proposed algorithm. Let $B^\infty_r(\bmx):=\big\{\bm y\in\R^d\mid \max_{\leqd} \vert y_i-x_i\vert <r\big\}$ denote the open ball of radius $r$ around $\bmx$ w.r.t.\ the supremum norm. We assume that we can simulate from finite PRMs $N_n$ with intensities $\mu\lc\cdot \cap \Ss_n\rc$, where $(\Ss_n)_{n\in\N}$ is a sequence of disjoint sets which satisfy 
    \begin{itemize}
        \item[(i)]  $\mu\lc \Ss_n\rc <\infty$ for all $n\in\N$,
        \item[(ii)]  $\bigcup_{n=1}^\infty \Ss_n =[0,\infty)^d\setminus \{\bm 0\}$,
        \item[(iii)]  for all $r>0$ there exists $m(r)\in\N$ such that $\bigcup_{n\geq m} \Ss_n\subset B^\infty_r\lc \bm 0\rc$. 
    \end{itemize} 
Under these conditions on $S_n$ we can propose the following algorithm for the exact simulation of max-id random vectors with exponent measure $\mu$.\\

\begin{algorithm}[H]
\label{alg2}
\SetAlgoLined \LinesNumbered
\KwResult{Unbiased sample of $(X_1,\ldots,X_d)$ with exponent measure $\mu$. }
Set $\mu_j\lc \cdot \rc=\mu\lc \cdot \cap  \big\{ \bmx\in [0,\infty)^d\setminus\{\bm 0\} \mid x_j>0,\ x_k=0,\ k<j, k\in J_0 \big\}\rc ,\ j\in J_0$\;
Set $\Tilde{\mu}\lc \cdot \rc=\mu\lc \cdot \cap  \big\{  \bmx\in [0,\infty)^d\setminus\{\bm 0\} \mid x_j=0 , j\in J_0 \big\}\rc$\;
\For{$j\in J_0$}{
Simulate a finite PRM $N_j$ with intensity $\mu_j$ and set $\hat{\bmX}_j=\max_{\bmx\in N_j}\bmx$\;
}
Set $\tilde{\bmX}=\bm 0$ and $n=1$\;
\While{there is no $r>0$ such that $\bigcup_{m\geq n} \Ss_{m}\subset B^\infty_r\lc \bm 0\rc $ \textbf{and} $\min_{i\not\in J_0} \Tilde{X}_i\geq r$}{
Simulate the finite PRM $\tilde{N}_n$ with intensity $\tilde{\mu}\lc \cdot\cap \Ss_n\rc$\;
Set $\tilde{\bmX}=\max \big\{ \max_{\bmx^{(n)}\in \tilde{N}_n} \bmx^{(n)},\tilde{\bmX}\big\}$\;
Set $n=n+1$\;
}

Set $\hat{\bmX}=\max \big\{ \max_{j\in J_0} \hat{\bmX}_j,\tilde{\bmX}\big\}$\;
\Return{$\hat{\bmX}$}
 \caption{Exact simulation of max-id random vector with vertex $\bm 0$} 
\end{algorithm}

\begin{thm}[Validity of Algorithm \ref{alg2}]
\label{thmalg2}
Algorithm \ref{alg2} stops after finitely many steps and its output is a max-id random vector $\bmX=(X_1,\ldots,X_d)$ with exponent measure $\mu$.
\end{thm}

The use of Algorithm \ref{alg2} is illustrated by the following example in which we provide an exact simulation algorithm for a large family of max-id distributions. As a byproduct, the simulation algorithm for max-stable random vectors \cite[Algorithm 2] {dombryengelkeoestingexactsimmaxstable2016} and the simulation algorithm for random vectors with reciprocal Archimedean copulas \cite[Algorithm 1] {maiexactsimreciparchim2018} are unified in a common simulation scheme. 

\begin{ex}[Common simulation scheme for max-stable random vectors with unit Fr\'echet margins and random vectors with reciprocal Archimedean copula]
\label{examplescalemixturesimplexponentmeasure}
Let $S_{\Vert\cdot\Vert}:=\{ \bmx\in[0,\infty)^d\mid \Vert \bmx\Vert =1\}$ denote the non-negative part of the unit sphere of some norm $\Vert\cdot\Vert$ on $\R^d$. Consider an exponent measure of the form
\begin{align}
\label{eqnscalemixtureexponentmeasure}
    \mu(A)&=\mu_1\otimes \mu_2\lc \big\{(m_1,\bm m_2)\in (0,\infty)\times S_{\Vert\cdot\Vert} \mid m_1\bm m_2\in A\big\}\rc, A\in\bc\lc [0,\infty)^d\rc,
\end{align}
where $\mu_2$ is a probability measure on $S_{\Vert\cdot\Vert}$ and $\mu_1$ is a measure on $(0,\infty)$ which satisfies $\mu_1\big( [r,\infty)\big)<\infty$ for all $r>0$. Setting $\mu_1=s^{-2}\rmd s$ and $\Vert\cdot\Vert=\Vert\cdot\Vert_1$ yields the family of max-stable distributions with unit Fr\'echet margins \cite[Section 5]{resnickextreme2013},  whereas setting $\Vert\cdot\Vert=\Vert\cdot\Vert_1$ and $\mu_2$ to the uniform distribution on $S_{\Vert\cdot\Vert_1}$ yields the family of distributions with reciprocal Archimedean copula and marginal distribution function $\exp\big( -\mu_1((\cdot,\infty))\big)$ \cite{maiexactsimreciparchim2018}.

Let $\lc E_j\rc_{j\in\N}$ denote a sequence of iid exponential random variables and let $\lc \bm Y_i\rc_\iinn$ denote a sequence of iid random vectors with distribution $\mu_2$ (independent of $\lc E_j\rc_{j\in\N}$). It is well known that the standard Poisson point process on $[0,\infty)$ with unit intensity can be represented as $\suminf \delta_{\sum_{j=1}^i E_j}$. Denoting $\mu_1^{\leftharpoonup}(t):=\sup\{s\in (0,\infty) \mid \mu_1\lc [s,\infty)\rc \geq t\}$ it is easy to see that \cite[Proposition 3.7]{resnickextreme2013} implies that $\suminf\delta_{\mu_1^{\leftharpoonup}\lc \sum_{j=1}^i E_j\rc }$ is a PRM with intensity $\mu_1$. Moreover, \cite[Proposition 3.8]{resnickextreme2013} implies that 
$$N=\suminf\delta_{\lc \mu_1^{\leftharpoonup}\lc \sum_{j=1}^i E_j\rc , \bm Y_i\rc }$$
is a PRM with intensity $\mu_1\otimes \mu_2$. Therefore, simulating $N\lc \cdot\cap \lc  [r,\infty)\times S_{\Vert\cdot\Vert}\rc\rc$ is achieved by iteratively simulating the iid random vectors $(E_i,\bm Y_i)$ until $\mu_1^{\leftharpoonup}\lc \sum_{j=1}^i E_j\rc < r$. Note that this only requires the simulation of finitely many random vectors since $\mu_1\lc [r,\infty)\rc<\infty$. Choosing $\Ss_n= \big[\frac{1}{n},\frac{1}{n-1}\big)\times S_{\Vert\cdot\Vert}$ one can easily check that the $\Ss_n$ satisfy all the required constraints. Therefore, Algorithm \ref{alg2} can be applied to exponent measures of the form (\ref{eqnscalemixtureexponentmeasure}). The stopping criterion of Algorithm \ref{alg2} depends on the chosen norm $\Vert \cdot \Vert$, but if $\Vert \cdot \Vert=\Vert \cdot \Vert_p$ for some $p\geq 1$, it is easy to see that the algorithm stops at least as soon as $\mu_1^{\leftharpoonup}\lc \sum_{j=1}^i E_j\rc<\min_{i\not\in J_0} \Tilde{X}_i$. 
\end{ex}

\begin{rem}
Example \ref{examplescalemixturesimplexponentmeasure} is easily extended to exponent measures of the form 
$$\mu(A)=\int_0^\infty \int_S \id_{\{m_1\bm m_2\in A\}}  K(m_1,\rmd \bm m_2)\rmd\mu_1 (m_1),$$
where $S$ denotes a bounded subset of $[0,\infty)^d\setminus \{0\}$ and $K(\cdot,\cdot)$ is a Markov kernel which satisfies $K(m_1,S)=1$ $\mu_1$-almost surely. For examples such ideas are used in \cite[Section 3.3]{huser2021maxidmodels} to construct the finite dimensional distributions of a spatial max-id process.
\end{rem}

\section{Proofs and technical Lemmas}
\label{apptechproofs}

\begin{lem}[L\'evy measure of stochastic integral w.r.t.\ id process]
\label{lemlevymeasurestochasticintegration}
Let $f: [0,\infty)\times[0,\infty)\to[0,\infty)$ denote a measurable function such that $f(s,\cdot)$ is non-decreasing and right-continuous for all $s\in[0,\infty)$. Let $\lc H_t\rc_{t\geq 0}$ denote a non-negative c\`adl\`ag id-process of bounded variation on compact sets with L\'evy measure $\nu$ and drift $b$. Moreover, assume that $0\leq H_f(t)(\omega):=\int_0^\infty f(s,t)\rmd H_s (\omega) <\infty$ for all $t\geq 0$ and $\omega\in\Omega$ and that the conditions of \cite[Theorem 2.7]{rajput1989spectral} are satisfied. Then
$$  \big( H_f(t)\big)_\tnn=\lc\int_0^\infty f(s,t)\rmd H_s\rc_{t\geq 0} $$
defines a nnnd c\`adl\`ag id-process with L\'evy measure
$$\nu_f\lc A\rc = \nu\lc \bigg\{ x\in D\big( [0,\infty)\big) \ \bigg\vert\ \int_0^\infty f(s,\cdot)\rmd x(s) \in A \text{ and } \int_0^\infty f(s,\cdot)\rmd x(s) \not= \bm 0 \bigg\} \rc,  $$
$A\in \bc\lc D([0,\infty))\rc$ and drift $b_f(t):=\int_0^\infty f(s,t)b(\rmd s)$.
\end{lem}

\begin{proof}[\normalfont{\textbf{Proof of Lemma  \ref{lemlevymeasurestochasticintegration}}}]
Well-definedness follows from the conditions of \cite[Theorem 2.7]{rajput1989spectral}. Infinite divisibility is obvious. The c\`adl\`ag property of $H_f$ follows from $H_f(t)<\infty$ for all $t>0$ and the non-decreasingness and right-continuity of $f(s,\cdot)$.
Since $\nu$ is $\sigma$-finite \cite[Proposition 2.10]{idprocessesrosinski2018} implies that there exists a PRM $M=\sum_{i\in\N} \delta_{x_i}$ with intensity $\nu$, such that 
$$\lc H_t\rc_{t\geq 0}\sim \lc \int x(t)M(\rmd x) +b(t) \rc_{\tnn}=\lc \sum_\iinn x_i(t)+b(t)\rc_{\tnn}.$$ 
Moreover, $M$, $b$ and $\nu$ can be chosen to be concentrated on the space of non-negative c\`adl\`ag functions of bounded variation on compact sets \cite[Theorem 3.4]{idprocessesrosinski2018}, denoted as $BV^+_r$. Therefore,
\begin{align*}
  \lc H_f(t)\rc_{t\geq 0} &\sim   \lc \int_0^\infty f(s,t) \lc \nlim\sum_{i=1}^n x_i \rc(\rmd s)+\int_0^\infty f(s,t)b(\rmd s)\rc_{t\geq 0} \\
  &= \lc \nlim \int_0^\infty f(s,t) \lc \sum_{i=1}^n x_i \rc(\rmd s)+\int_0^\infty f(s,t)b(\rmd s)\rc_{t\geq 0}\\
  &= \lc \sum_{\substack{ i\in\N \\ \int_0^\infty f(s,\cdot) x_i(\rmd s)\not=0} } \int_0^\infty   f(s,t) x_i(\rmd s)+\int_0^\infty f(s,t)b(\rmd s)\rc_{t\geq 0}=  \lc \int_{BV^+_r}  \tilde{x}(t)d\tilde{M}(\tilde{x})+\int_0^\infty f(s,t)b(\rmd s)\rc_{t\geq 0},
\end{align*}
where $\tilde{M}:=\sum_{ i\in\N ,\ \int_0^\infty f(s,\cdot) x_i(\rmd s)\not=0} \delta_{\int_0^\infty f(s,\cdot)x_i(\rmd s)}$ denotes a PRM on $D([0,\infty))$ with intensity $\nu_f$, since the map $x\mapsto \int_0^\infty f(s,\cdot)x(\rmd s)$ is measurable in $D\big([0,\infty)\big)$ equipped with the sigma-algebra generated by the finite dimensional projections. $\nu_f$ satisfies $\nu_f(\bm 0)=0$ and $\int_{BV^+_r} x(t)\rmd \nu_f(x)<\infty$ by the conditions of \cite[Theorem 2.7]{rajput1989spectral}. Thus, $\nu_f$ is a L\'evy measure and the lemma is proven.

\end{proof}


\begin{proof}[\normalfont{\textbf{Proof of Theorem \ref{thmvalalg1}}}]
By (\ref{eqndefmuj}), $\mu_j$ is a finite measure for each $j\in J_0$. Thus, $\hat{\bmX}_j$ is obtained by the simulation of a finite PRM with intensity $\mu_j$. Therefore, Algorithm \ref{alg1} stops after finitely many steps if and only if the for-loop from lines 7-30 stops after finitely many steps. Thus, consider the setting of line 7 and let $\tilde{N}$ denote a PRM with intensity $\tilde{\mu}$ defined in (\ref{eqndefmutilde}). By the definition of $\tilde{\mu}$ we obtain that the associated max-id process $\tilde{\bmX}$ satisfies $\p(\Tilde{X}_{t_j}=0)=1$ for all $j\in J_0$ and $\p\lc \Tilde{X}_{t_i}>0\rc=1$ for all $i\not\in J_0$. Thus, if $\Tilde{X}_{t_i}=0$ and $i\not\in J_0$, there is almost surely some $c>0$ such that $\tilde{N}^+_{t_i}\subset \tilde{N}\lc \cdot \cap \{f\in C_0(T)\mid f(t_i)\geq c\}\rc$. Moreover, if $\Tilde{X}_{t_i}>0$ we get that $\tilde{N}^+_{t_i}\subset \tilde{N}(\cdot \cap \{f\in C_0(T)\mid f(t_i)\geq \Tilde{X}_{t_i}\}$ almost surely. Thus, the simulation of $\tilde{\bmX}$ only requires the simulation of PRMs with finite intensity measures and stops after finitely many steps.

It remains to prove that $\bmX_\bmt\sim\hat{\bmX}_\bmt$. Observe that
$\hat{\bmX}=\max\big\{  \max_{j\in J_0} \hat{\bmX}_j,\tilde{\bmX}\big\}$ in line 31 is the maximum of two independent stochastic processes. The first process $\max_{j\in J_0} \hat{\bmX}_j$ is an exact simulation of the sample path of a continuous max-id process with exponent measure $\mu\lc \cdot \cap \big( \cup_{j\in J_0}\{ f\in C_0(T)\mid f(t_j)>0 \}\big) \rc$. The second process $\tilde{\bmX}$ is an exact simulation of $\max_{f\in \tilde{N}^+_{(t_i)_{i\not\in J_0}}} f$. Thus, $\hat{\bmX}_\bmt=\max\big\{ \max_{j\in J_0} \hat{\bmX}_j(\bmt); \tilde{\bmX}_\bmt \big\}$ is an exact simulation of a max-id random vector with exponent measure 
\begin{align*}
    \sum_{j\in J_0} &\mu_j\lc \lc f(t_1),\ldots,f(t_d)\rc\in \cdot\rc+ \tilde{\mu}\lc \{ \lc f(t_1),\ldots,f(t_d)\rc\in \cdot\}\rc\\
    &=\mu\bigg( \lc f(t_1),\ldots,f(t_d)\rc\in \cdot \cap \bigg(\bigcup_{j\in J_0}\{  f(t_j)>0 \}\bigg)\bigg) + \mu\lc \{ \lc f(t_1),\ldots,f(t_d)\rc\in \cdot \cap \{  f(t_j)=0 \ \forall\  j\in J_0 \} \}\rc\\
    &=\mu\lc \{ \lc f(t_1),\ldots,f(t_d)\rc\in \cdot\}\rc,
\end{align*}
which is the exponent measure of $\bmX_\bmt$ and shows that $\hat{\bmX}_t\sim\bmX_t$.
\end{proof}

\begin{proof}[\normalfont{\textbf{Proof of Lemma \ref{lemexpectedsizeextremalfunction}}}]
Recall that \cite[Appendix A.3]{dombryeyiminkoconddist2013} verifies that $N_{\bmt}^+$ and $N_{\bmt}^-$ are well-defined point measures. Thus, $\big\vert N_{\bmt}^+\big\vert$ is an $\N_0\cup\{\infty\}$-valued random variable and $\e\lk \big\vert N_{\bmt}^+\big\vert\rk$ is well defined. Following the ideas of \cite{oestingschlatherzhou2013,oestingschlatherzhou2018}, consider some $a>0$ and the set $A_a=\{f\in C_0(T)\mid f(t_i)\geq a\text{ for some }  1\leq i\leq d\}$. Then $\mu(A_a)<\infty$, $\vert N(A_a)\vert\sim Poi\lc\mu(A_a)\rc$ and
\begin{align*}
    \e\lk \big\vert  N_{\bmt}^+( A_a)\big\vert\rk&=\e\lk \big\vert N( A_a)\big\vert\rk 
    -\e\lk \big\vert  N_{\bmt}^-(A_a)\big\vert\rk =\int_{C_0(T)} \id_{\{f(t_i)\geq a\text{ for some }  1\leq i\leq d\}} \rmd \mu(f) -\e\lk \e\lk \big \vert N_{\bmt}^-(A_a)\big\vert \ \bigg\vert\ N_{\bmt}^+ \rk \rk \\
    &=\int_{C_0(T)} \id_{\{f(t_i)\geq a\text{ for some }  1\leq i\leq d\}} \rmd \mu(f) -\e\lk \int_{C_0(T)} \id_{\{f(t_i)\geq a\text{ for some }  1\leq i\leq d\}} \id_{\{f(t_i)<X_{t_i}\text{ for all } 1\leq i\leq d\}} \rmd \mu(f) \rk\\
    &=\e\lk \int_{C_0(T)} \id_{\{f(t_i)\geq a\text{ for some }  1\leq i\leq d\}} \id_{\{f(t_i)\geq X_{t_i}\text{ for some } 1\leq i\leq d\}} \rmd \mu(f) \rk,
\end{align*}
where we used that $N_{\bmt}^-$ given $N_{\bmt}^+$ is distributed as a PRM with intensity $\id_{\{ f(t_i)< \max_{\tilde{f}\in N^+_\bmt}\tilde{f}(t_i) \}}\rmd\mu(f)$.
We conclude by considering three cases:
\begin{enumerate}
    \item[(i)] Assume that $\p\lc X_{t_j}>0\rc=1$ for all $1\leq i\leq d$. When $a\searrow 0$ the monotone convergence theorem implies that
\begin{align*}
    \e\lk \big\vert  N_{\bmt}^+\big\vert\rk =
    \e\lk \int_{C_0(T)} \id_{\{f(t_i)\geq X_{t_i}\text{ for some } 1\leq i\leq d\}} \rmd \mu(f) \rk.
\end{align*}
\item[(ii)]  If $\p\lc X_{t_j}=0\rc>0$ for some $1\leq i\leq d$ and $\mu$ is a finite measure then one may take $a=0$ which immediately implies 
\begin{align*}
    \e\lk \big\vert  N_{\bmt}^+\big\vert\rk =
    \e\lk \int_{C_0(T)} \id_{\{f(t_i)\geq X_{t_i}\text{ for some } 1\leq i\leq d\}} \rmd \mu(f) \rk.
\end{align*}
\item[(iii)]  If $\p\lc X_{t_j}=0\rc>0$ for some $1\leq i\leq d$ and $\mu$ is an infinite measure then 
$$\e\lk \big\vert   N_{\bmt}^+\big\vert\rk=\infty=\e\lk \int_{C_0(T)} \id_{\{f(t_i)\geq X_{t_i}\text{ for some } 1\leq i\leq d\}} \rmd \mu(f) \rk ,$$
since $\p\lc \big\vert N_{\bmt}^+\big\vert=\infty\rc\geq \p\lc X_{t_i}=0 \text{ for some }1\leq i\leq d\rc>0.$ 
\end{enumerate}

\end{proof}

\begin{proof}[\normalfont{\textbf{Proof of Theorem \ref{thmcomplexityofalg1}}}]
Obviously, the expected number of simulated functions (atoms) of the PRMs with intensities $\lc \mu_j\rc_{j\in J_0}$ is 
\begin{align*}
    &\sum_{j\in J_0} \mu_j\lc C_0(T) \rc=\sum_{j\in J_0}\mu\lc \cdot \cap  \Big\{ f\in C_0(T) \mid f(t_j)>0, f(t_k)=0,\ k<j,\ k\in J_0\Big \}\rc \\
    &=\mu\lc \Big\{ f\in C_0(T) \mid f(t_j)>0 \text{ for some } j\in J_0\Big \} \rc.
\end{align*}
Thus, the expected number of functions that need to be simulated to obtain $\lc \hat{\bmX}_j\rc_{j\in J_0}$  is equal to $\mu\Big( \big\{ f\in C_0(T) \mid f(t_j)>0$ $ \text{ for some } j\in J_0\big \} \Big)$.

It remains to compute the expected number of simulated functions to obtain $\Tilde{\bmX}$.
At each location $(t_{i})_{i\not\in J_0}$, according to Algorithm \ref{alg1} and assumption (\ref{assumptionsimprm}), we can consecutively simulate the atoms $f^{(i)}_1,f^{(i)}_2,\ldots$ of a PRM $\tilde{N}^{(i)}$ with intensity $\tilde{\mu}$ such that $f^{(i)}_1(t_i)\geq f^{(i)}_2(t_i)\geq\cdots$ until the first subextremal function is found. Since all simulated atoms which satisfy $f^{(i)}_j(t_k)\geq \Tilde{X}_{t_k}$ for some $k<i$, $k\not\in J_0$, are rejected we obtain that the number of functions that need to be simulated to obtain $\Tilde{\bmX}$ is
$$ \Big\vert \tilde{N}_{(t_i)_{i\not\in J_0}}^+ \Big\vert + \sum_{i\not\in J_0} \lc \Big\vert \Big\{ f^{(i)}_j \ \big\vert \ f^{(i)}_j(t_k)\geq X_{t_{i_k}}\text{ for some } ,k\not\in J_0, k<i;\ f^{(i)}_j(t_i)\geq X_{t_{i}} \Big\} \Big\vert +1\rc . $$
Note that the number of rejected functions is increased by $1$, since we have to simulate until the first subextremal function at each location $t_{i}$ is obtained. Thus, the expected number of functions that need to be simulated to obtain $\tilde{\bmX}$ is given by
\begin{align*}
\e\lk \big\vert \tilde{N}_{(t_i)_{i\not\in J_0}}^+ \big\vert \rk + d-\vert J_0\vert + \sum_{i\not\in J_0} \e\lk \Big\vert \Big\{ f^{(i)}_j \ \Big\vert \ f^{(i)}_j(t_k)\geq \Tilde{X}_{t_k}\text{ for some } k\not\in J_0, k<i;\ f^{(i)}_j(t_i)\geq \Tilde{X}_{t_k} \Big\} \big\vert\rk .
\end{align*}
The expectation of the first term is provided by Lemma \ref{lemexpectedsizeextremalfunction}. Thus we focus in the remaining expectation and obtain
\begin{align*}
    &\e\lk \Big\vert \Big\{ f^{(i)}_j \ \big\vert \ f^{(i)}_j(t_k)\geq \Tilde{X}_{t_k}\text{ for some } k\not\in J_0, k<i;\ f^{(i)}_j(t_i)\geq \Tilde{X}_{t_i} \Big\} \Big\vert\rk\\
    &=\e\lk \Big\vert \Big\{ f^{(i)}_j \ \big\vert \ f^{(i)}_j(t_k)\geq \Tilde{X}_{t_k}\text{ for some } k\not\in J_0, k<i;\ f^{(i)}_j(t_i)\geq \Tilde{X}_{t_i} \Big\} \Big\vert\ \bigg\vert\ \tilde{N}_{(t_k)_{k\not\in J_0, k<i}}^+ , \Big\{ f^{(i)}_j \ \big\vert \ f^{(i)}_j(t_k)<\Tilde{X}_{t_k}\text{ for all } k\in J_0,k<i\Big\} \rk .
\end{align*}
Note that $\Big\{ f^{(i)}_j \ \big\vert \ f^{(i)}_j(t_k)<\Tilde{X}_{t_k}\text{ for all } k\not\in J_0,k<i\Big\}$ and $\Big\{ f^{(i)}_j \ \big\vert \ f^{(i)}_j(t_k)\geq \Tilde{X}_{t_k}\text{ for some } k\not\in J_0, k<i \Big\}$ are disjoint measurable sets and therefore, conditioned on $\lc X_{t_k}\rc_{k< i,k\not\in J_0}$, the restrictions of the PRM $\tilde{N}^{(i)}$ on each of the two sets are independent PRMs with intensities $\id_{\{ f^{(i)}_j \ \vert \ f^{(i)}_j(t_k)<\Tilde{X}_{t_k}\text{ for all } k\in J_0,k<i\}}$ and $\id_{\{ f^{(i)}_j \ \vert \ f^{(i)}_j(t_k)\geq \Tilde{X}_{t_k}\text{ for some } k\not\in J_0, k<i \}}$. Moreover, since $ \tilde{N}_{(t_k)_{k\not\in J_0, k<i}}^+ $ and $ \Big\{ f^{(i)}_j \ \big\vert \ f^{(i)}_j(t_k)<\Tilde{X}_{t_k}\text{ for all } k\in J_0,k<i\Big\}$ determine $\lc\Tilde{X}_{t_k}\rc_{k< i,k\not\in J_0}$ and $\Tilde{X}_{t_i}$ we get
\begin{align*}
    &\e\lk \Big\vert \Big\{ f^{(i)}_j \ \big\vert \ f^{(i)}_j(t_k)\geq \Tilde{X}_{t_k}\text{ for some } k\not\in J_0, k<i;\ f^{(i)}_j(t_i)\geq \Tilde{X}_{t_i} \Big\} \Big\vert\ \bigg\vert\ \tilde{N}_{(t_k)_{k\not\in J_0, k<i}}^+ , \Big\{ f^{(i)}_j \ \big\vert \ f^{(i)}_j(t_k)<\Tilde{X}_{t_k}\text{ for all } k\in J_0,k<i \Big\} \rk \\
    &= \e\lk \int_{C_0(T)} \id_{\{ f(t_k)\geq \Tilde{X}_{t_k}\text{ for some } k\not\in J_0,  k<i ;\ f(t_i)\geq \Tilde{X}_{t_i} \}} \rmd \tilde{\mu}(f) \rk \\
    &=\e\lk \int_{C_0(T)} \id_{\{ f(t_k)\geq \Tilde{X}_{t_k}\text{ for some } k\not\in J_0,k<i \}} \rmd \tilde{\mu}(f) \rk - \e\lk \int_{C_0(T)} \id_{\{ f(t_k)\geq \Tilde{X}_{t_k}\text{ for some } k\not\in J_0,  k<i ;\ f(t_i)< \Tilde{X}_{t_i} \}} \rmd \tilde{\mu}(f) \rk .
\end{align*}
Now, Lemma \ref{lemexpectedsizeextremalfunction} implies
\begin{align*}
    & \e\lk \int_{C_0(T)} \id_{\{ f(t_k)\geq \Tilde{X}_{t_k}\text{ for some } k\not\in J_0,k<i \}} \rmd \tilde{\mu}(f) \rk - \e\lk \int_{C_0(T)} \id_{\{ f(t_k)\geq \Tilde{X}_{t_k}\text{ for some } k\not\in J_0,  k<i ;\ f(t_i)< \Tilde{X}_{t_i} \}} \rmd \tilde{\mu}(f) \rk \\
    &= \e\lk \big\vert \tilde{N}_{(t_k)_{k\not\in J_0},k<i}^+ \big\vert \rk -  \e\lk \int_{C_0(T)} \lc \id_{\{ f(t_k)\geq \Tilde{X}_{t_k}\text{ for some } k\not\in J_0,  k\leq i  \}} 
    - \id_{\{ f(t_k)< \Tilde{X}_{t_k}\text{ for all } k\not\in J_0,  k< i ; f(t_i)\geq \Tilde{X}_{t_i} \}}\rc \lc 1-\id_{\{f (t_i)\geq \Tilde{X}_{t_i}\}}\rc \rmd \tilde{\mu}(f) \rk \\
    &= \e\lk \big\vert \tilde{N}_{(t_k)_{k\not\in J_0},k<i}^+ \big\vert \rk -  \e\lk \int_{C_0(T)} \id_{\{ f(t_k)\geq \Tilde{X}_{t_k}\text{ for some } k\not\in J_0,  k\leq i  \}} 
     \lc 1-\id_{\{f (t_i)\geq \Tilde{X}_{t_i}\}}\rc \rmd \tilde{\mu}(f) \rk \\
    &=\e\lk \big\vert \tilde{N}_{(t_k)_{k\not\in J_0},k<i}^+ \big\vert \rk -  \e\lk \big\vert \tilde{N}_{(t_k)_{k\not\in J_0},k\leq i}^+ \big\vert \rk +\e\lk \int_{C_0(T)} \id_{\{ f(t_k)\geq \Tilde{X}_{t_k}\text{ for some } k\not\in J_0,  k\leq i  \}} \id_{\{ f(t_i)\geq \Tilde{X}(t_i),f(t_k)\geq \Tilde{X}_{t_k} \text{ for some }k\in J_0,k\leq i \}} \rmd \tilde{\mu}(f) \rk \\
    &=\e\lk \big\vert \tilde{N}_{(t_k)_{k\not\in J_0},k<i}^+ \big\vert \rk -  \e\lk \big\vert \tilde{N}_{(t_k)_{k\not\in J_0},k\leq i}^+ \big\vert \rk+\e\lk \int_{C_0(T)}  \id_{\{ f(t_i)\geq \Tilde{X}(t_i)\}} \rmd \tilde{\mu}(f) \rk \\
    &=\e\lk \big\vert \tilde{N}_{(t_k)_{k\not\in J_0},k<i}^+ \big\vert \rk -  \e\lk \big\vert \tilde{N}_{(t_k)_{k\not\in J_0},k\leq i}^+ \big\vert \rk+\e\lk  \tilde{\mu}\lc \big\{ f\in C_0(T)\mid f(t_i)\in \big[\Tilde{X}_{t_i},\infty\big) \big\} \rc \rk
\end{align*}
Thus,
\begin{align}
&\e\lk \big\vert \tilde{N}_{(t_i)_{i\not\in J_0}}^+ \big\vert \rk + \sum_{i\not\in J_0} \e\lk \big\vert \Big\{ f^{(i)}_j \ \big\vert \ f^{(i)}_j(t_k)\geq \Tilde{X}_{t_k}\text{ for some } k\not\in J_0, k<i;\ f^{(i)}_j(t_i)\geq \Tilde{X}_{t_k} \Big\} \big\vert\rk+ d-\vert J_0\vert  \nonumber \\
&= \e\lk \big\vert \tilde{N}_{(t_i)_{i\not\in J_0}}^+ \big\vert \rk + \sum_{i=1,\ldots,n, i\not\in J_0}  \e\lk \big\vert \tilde{N}_{(t_k)_{k\not\in J_0},k<i}^+ \big\vert \rk -  \e\lk \big\vert \tilde{N}_{(t_k)_{k\not\in J_0},k\leq i}^+ \big\vert \rk+\e\lk  \tilde{\mu}\lc \big\{ f\in C_0(T)\mid f(t_i)\in \big[\Tilde{X}_{t_i},\infty\big) \big\} \rc \rk +d-\vert J_0\vert  \nonumber \\
&=d-\vert J_0\vert +\sum_{ i\not\in J_0} \e\lk  \tilde{\mu}\lc \big\{ f\in C_0(T)\mid f(t_i)\in \big[\Tilde{X}_{t_i},\infty\big) \big\}\rc \rk  \label{eqnexpnmboffcts}
\end{align}

If $\Tilde{X}_{(t_i)_{i\not\in J_0}}$ has continuous marginal distribution we can stop as soon as we found an extremal function at each location $(t_i)_{i\not \in J_0}$. Therefore, the term $d-\vert J_0\vert$ which comes from the simulation of the first subextremal function may be omitted from (\ref{eqnexpnmboffcts}) and we get
\begin{align*}
&\sum_{ i\not\in J_0} \e\lk  \tilde{\mu}\lc \big\{ f\in C_0(T)\mid f(t_i)\in \big[\Tilde{X}_{t_i},\infty\big) \big\} \rc \rk =\sum_{ i\not\in J_0} \e\lk  \tilde{\mu}\lc \big\{ f\in C_0(T)\mid f(t_i)\in \big(\Tilde{X}_{t_i},\infty\big) \big\} \rc \rk =\sum_{ i\not\in J_0} \e\lk  -\log\lc 1- F_{t_i}\lc \Tilde{X}_{t_i}\rc \rc \rk\\
&=d-\vert J_0\vert 
\end{align*}
where $F_t(x):=\p(\Tilde{X}_t\leq x)$ denotes the marginal distribution function of $\Tilde{X}_t$ and it is well known that $-\log\lc 1- F_{t_i}(\Tilde{X}_{t_i})\rc$ $\sim\ $Exp$(1)$, since $F_{t_i}(\Tilde{X}_{t_i})$ is uniformly distributed on $[0,1]$ when $\Tilde{X}_t$ follows a continuous distribution.
Combing the results above we obtain the claimed complexity of Algorithm \ref{alg1}.

\end{proof}

\begin{proof}[\normalfont\textbf{Proof of Lemma \ref{lemlevymeasureselfsimbdlp}}]
Note that the L\'evy measure $\nu_k$ of $L^{(k)}$ on $\mathbf{M}$ is given by the image measure of the map 
$$ \big(\lc (0,\infty),\lambda_{0,\infty}\rc \times \lc (0,\infty),\rho_k\rc\big) \to \mathbf{M} ;\ (s,a)\mapsto a\id_{\{\cdot\geq s\}}, $$
where $\lambda_{0,\infty}$ denotes the Lebesgue measure on $(0,\infty)$ and $\rho_k$ denotes the univariate L\'evy measure of $L^{(k)}_1$. To derive the path L\'evy measure of the self similar subordinator $\hto$ we first need to derive the path L\'evy measures of the two independent id processes $\hat{L}^{(1)}_t:=\int_{-\log\lc \min\{t; 1\}\rc}^\infty \exp(-s)\rmd L^{(k,1)}_s$ and $\hat{L}^{(2)}_t:= \int_0^{\log\lc \max\{1,t\}\rc} \exp(s)\rmd L^{(k,2)}_s$.
Note that for $t\in[0,1]$  
$$ \int_{-\log(t)}^\infty   \exp(-y) \lc a\id_{\{\cdot \geq s\}}\rc(\rmd y) =a\exp(-s)\id_{\{s\geq -\log(t )\}} $$
and for $t>1$ 
$$ \int_{0}^{\log(t)}    \exp(y) \lc a\id_{\{\cdot \geq s\}}\rc(\rmd y) =a\exp(s)\id_{\{s\leq \log(t)\}} .$$ 
An application of Lemma \ref{lemlevymeasurestochasticintegration} shows that the L\'evy measure of $\hat{L}^{(1)}$ is given by
$$\nu_1(A)=\lambda_{0,\infty}\otimes \rho_k  \lc \bigg\{ (s,a)\ \bigg\vert\  a\exp(-s)\id_{\{s\geq -\log(\min\{\cdot,1\} )\}}  \in A  \bigg\} \rc;\ A\in\bc({\mathbf{M}}) $$
and that the L\'evy measure of $\hat{L}^{(2)}$ is given by
$$\nu_2(A)=\lambda_{0,\infty}\otimes \rho_k  \lc \bigg\{ (s,a)\ \bigg\vert\  a\exp(s )\id_{\{s\leq \log(\max\{\cdot,1\} )\}}  \in A  \bigg\} \rc ;\ A\in\bc({\mathbf{M}}).$$
This implies that the path L\'evy measure of the self-similar subordinator $H$ is given by $\nu=\nu_1+\nu_2$, since $\hat{L}^{(1)}$ and  $\hat{L}^{(2)}$ are independent. It remains to verify (\ref{pathlevymeasureselfsimBDLP}). To this purpose we simply verify that the Laplace transform of $H$ coincides with the Laplace transform of an id process with path L\'evy measure (\ref{pathlevymeasureselfsimBDLP}), since a path L\'evy measure is unique. 
\begin{align*}
    \e\lk \exp\lc \sum_{i=1}^d z_iH(t_i)\rc\rk&  =  \exp\lc \bigintssss_0^\infty \bigintssss_0^\infty \lc 1-\exp\lc -  \sum_{i=1}^d z_i a\exp(-s) \id_{\{s\geq -\log(\min\{t_i,1\} )\}} \rc\rc \rmd s \rho_k(\rmd a) \rc \\
    &+  \exp\lc \bigintssss_0^\infty\bigintssss_0^\infty \lc 1-\exp\lc -  \sum_{i=1}^d z_i a\exp(s)\id_{\{s\leq \log(\max\{t_i,1\})\}} \rc\rc\rmd s \rho_k(\rmd a)  \rc \\
    &=\exp\lc \bigintssss_0^\infty \bigintssss_0^1 \lc 1-\exp\lc -  \sum_{i=1}^d z_i as\id_{\{-\log(s)\geq -\log(\min\{t_i,1\} )\}} \rc\rc s^{-1}  \rmd s \rho_k(\rmd a) \rc\\
    &+ \exp\lc \bigintssss_0^\infty\bigintssss_1^\infty \lc 1-\exp\lc -  \sum_{i=1}^d z_i as\id_{\{\log(s)\leq \log(\max\{t_i,1\})\}} \rc\rc s^{-1}\rmd s \rho_k(\rmd a)  \rc \\
    &=\exp\lc \bigintssss_0^\infty \bigintssss_0^\infty \lc 1-\exp\lc -  \sum_{i=1}^d z_i as\id_{\{t_i\geq s\}} \rc\rc s^{-1} \rmd s \rho_k(\rmd a) \rc
\end{align*}
\end{proof}

\begin{rem}[Path L\'evy measure of general self-similar processes]
The path L\'evy measure representation in (\ref{pathlevymeasureselfsimBDLP}) is not only valid for nnnd self-similar processes but also valid for general self-similar processes where $\rho_k$ denotes the L\'evy measure of the BDLP of $H_1$. Moreover, since a self-similar process with index $\gamma>0$ corresponds to a time change of a self-similar process with index $1$, the path L\'evy measure $\nu^{(\gamma)}$ of a self-similar process with index $\gamma$ is simply obtained by applying the same ``time change'' to the L\'evy measure of the self-similar process with index $1$, i.e.\ by the image measure of $({\mathbf{M}},\nu)\to {\mathbf{M}},\ \lc f(t)\rc_{t\geq 0}\mapsto \lc f(t^\gamma)\rc_{t\geq 0}$.
\end{rem}

\begin{proof}[\normalfont{\textbf{Proof of Theorem \ref{thmalg2}}}]
The $\mu_j$ are finite intensity measures by their definition in (\ref{eqndefmuj}). Therefore, Algorithm \ref{alg2} stops after finitely many steps if and only if the while-loop from lines 7-11 stops after finitely many steps. It is obvious that the simulation of each PRM $\tilde{N}_n$ only requires the simulation of finitely many points.  Thus, we need to check that the condition $C:=\big\{$\text{there is no $r>0$ such that $ \cup_{m\geq n} \Ss_{m}\subset B^\infty_r\lc \bm 0\rc $ \textbf{and} $\min_{i\not\in J_0} \Tilde{X}_i\geq r\big\}$} is violated after finitely many steps. Let $\tilde{N}$ denote the PRM with intensity $\tilde{\mu}$. It is easy to see that condition $C$ is eventually violated after finitely many steps if and only if $\min_{i\not\in J_0,\bmx\in \tilde{N}} x_i>0$ almost surely. By the construction of $\tilde{\mu}$ we have $\p( \Tilde{X}_i=0)=0$ for all $i\not\in J_0$, which implies that $\min_{i\not\in J_0,\bmx\in \tilde{N}} x_i>0$ almost surely and the algorithm stops after finitely many steps. 

It remains to prove that $\hat{\bmX}\sim \bmX$. Clearly, if condition $C$ is violated for some $n\in\N$ and $r>0$, then all points of the PRM $\tilde{N}$ in $ \cup_{m<n} \Ss_{m}$ have already been simulated and $\tilde{\bmX}=\max_{\bmx\in \tilde{N}\lc\cdot\cap\lc \cup_{m< n} \Ss_{m}\rc\rc}\bmx$. A point $\bmx \in \tilde{N}\lc \cdot \cap \lc\cup_{m\geq n} \Ss_{m}\rc\rc$ can only increase a non-zero component $( \Tilde{X}_j)_{j\not\in J_0}$ of $\tilde{\bmX}$ if $\max_{i\not\in J_0} x_j\geq r$. However, since $\cup_{m\geq n} \Ss_{m}\subset B^\infty_r$, we actually have that $\tilde{\bmX}=\max_{\bmx\in \tilde{N}} \bmx$. Thus, $\tilde{\bmX}$ is max-id with exponent measure $\tilde{\mu}$. Combining this with the fact that the $\mu_j$ and $\tilde{\mu}$ are supported on disjoint sets, we obtain that $\hat{\bmX}$ is max-id with exponent measure $\sum_{j\in J_0} \mu_j+\tilde{\mu}=\mu$, which proves the claim.
\end{proof}

\section*{Acknowledgements}

I want to thank Jan-Frederik Mai for encouraging me to pursue the idea of deriving an exact simulation algorithm for exchangeable min-id sequences and repeatedly proofreading earlier versions of the manuscript. Moreover, I want to thank Matthias Scherer for repeatedly proofreading earlier versions of the manuscript. Their helpful comments largely improved the quality of the paper. Last but not least, I want to thank an anonymous referee for pointing out how to conduct a complexity analysis of the proposed simulation algorithm and another anonymous referee and the associate editor for their constructive comments which led to significant improvements of the paper.


\bibliographystyle{myjmva}
\bibliography{references.bib}

\end{document}